\documentclass[smallcondensed]{svjour3}

\smartqed  
\usepackage{graphicx}

\usepackage{amsmath,amsfonts,amssymb,theorem}
\usepackage{color}
\usepackage{enumerate}

\usepackage{pgf}

\makeatletter
\renewcommand{\fnum@figure}{Figure \thefigure}
\makeatother

\DeclareMathOperator{\tr}{tr} 
\DeclareMathOperator{\e}{e}   
\DeclareMathOperator{\sign}{sign}
\DeclareMathOperator{\grad}{grad}
\DeclareMathOperator{\dive}{div}

\newtheorem{thm}{Theorem}
\newtheorem{prop}[thm]{Proposition}
\newtheorem{lem}[thm]{Lemma}

\newcommand{\R}{{\mathbb R}}
\newcommand{\dd}[2]{\frac{\text{d} #1}{\text{d} #2}}

\newcommand{\cX}{{\sf X}}
\newcommand{\cY}{{\sf Y}}
\newcommand{\cZ}{{\sf 0}}
\newcommand{\cS}{{\sf S}}
\newcommand{\cP}{{\sf P}}

\newcommand{\al}{{\alpha}}
\newcommand{\be}{{\beta}}

\newcommand{\cdts}{\hspace{1ex} \cdots \hspace{1ex}}

\begin{document}

\title{%
On global stability of the Lotka reactions with generalized mass-action kinetics
}

\author{%
Bal\'azs~Boros \and Josef~Hofbauer \and Stefan~M\"uller
}

\institute{%
Bal\'azs Boros and Stefan~M\"uller \at
Radon Institute for Computational and Applied Mathematics, \\
Austrian Academy of Sciences,
Altenbergerstrasse 69, 4040 Linz, Austria \\
\email{balazs.boros@ricam.oeaw.ac.at, stefan.mueller@ricam.oeaw.ac.at} 
\and
Josef Hofbauer \at
Faculty of Mathematics, University of Vienna, \\ Oskar-Morgenstern-Platz 1, 1090 Vienna, Austria \\
\email{josef.hofbauer@univie.ac.at}
}

\date{Received: date / Accepted: date}

\maketitle

\begin{abstract}
\noindent
Chemical reaction networks with generalized mass-action kinetics lead to power-law dynamical systems.
As a simple example, we consider the Lotka reactions 
with two chemical species and arbitrary power-law kinetics.
We study existence, uniqueness, and stability of the positive equilibrium,
in particular, we characterize its global asymptotic stability in terms of the kinetic orders. 
\keywords{chemical reaction network \and power-law kinetics \and Andronov-Hopf bifurcation \and Dulac function}
\end{abstract}


\section{Introduction}

Lotka~\cite{lotka:1910} considered a series of chemical reactions, 
$\cS \to \cX$, $\cX \to \cY$, and $\cY \to \cP$, which transform a substrate into a product.
If the first and the second reaction are assumed to be autocatalytic in $\cX$ and $\cY$, respectively,
then the resulting dynamics is equivalent to the classical Lotka-Volterra predator-prey system~\cite{lotka:1920:a,lotka:1920:b}.

Farkas and Noszticzius \cite{farkas:noszticzius:1985} and Dancs\'o et al.~\cite{dancso:farkas:farkas:szabo:1991}
considered generalized Lotka-Volterra schemes,
that is, the Lotka reactions with power-law kinetics (and non-negative real exponents),
and provided a local stability analysis 
as well as first integrals.
We continue this line of research and consider the Lotka reactions with {\em arbitrary} power-law kinetics.
Moreover, we provide a {\em global} stability analysis.

As in the example of the Lotka reactions,
every chemical reaction network with generalized mass-action kinetics (power-law kinetics)
leads to a generalized mass-action system.
For such systems,
M\"uller and Regensburger~\cite{mueller:regensburger:2012,mueller:regensburger:2014} analyzed existence and uniqueness of (special) equilibria.
The works~\cite{farkas:noszticzius:1985,dancso:farkas:farkas:szabo:1991} and the present work are first steps
towards an understandig of their stability. 

In a classical Lotka-Volterra system, the unique positive equilibrium is neutrally stable,
and all other positive solutions are periodic, corresponding to closed orbits.
Generically, a Lotka-Volterra scheme still has a unique positive equilibrium, 
but it can be stable or unstable.
The corresponding Andronov-Hopf bifurcation can be supercritical, subcritical, or degenerate
leading to asymptotically stable, repelling, or a continuum of closed orbits. 

As our main result,
we characterize global asymptotic stability of the unique positive equilibrium in terms of the kinetic orders.
Thereby, we apply the Bendixson-Dulac test in order to rule out periodic solutions.
Further, we rule out unbounded solutions and solutions approaching the boundary of the positive quadrant.

The paper is organized as follows.
In Section 2, we motivate generalized Lotka-Volterra schemes as chemical reaction networks with generalized mass-action kinetics.
In Section 3, we present our main results,
and in Section 4, we provide the corresponding proofs.


\section{Lotka-Volterra schemes as generalized mass-action systems}

As in the original work by Lotka \cite{lotka:1910},
we start by considering a series of irreversible chemical reactions,
\[
\cS \longrightarrow \cX, \quad
\cX \longrightarrow \cY, \quad
\cY \longrightarrow \cP.
\]
The first reaction turns a substrate into species $\cX$, the second reaction transforms $\cX$ into $\cY$,
and the third reaction turns species $\cY$ into a product or degrades it.
We are interested in the dynamics of $\cX$ and $\cY$ only,
in particluar, we assume that the substrate is present in constant amount and that the product does not affect the dynamics.
As a consequence, we can omit substrate and product from consideration, and obtain the simplified reactions
\[
\cZ \longrightarrow \cX, \quad
\cX \longrightarrow \cY, \quad
\cY \longrightarrow \cZ.
\]

To obtain a classical Lotka-Volterra system as in~\cite{lotka:1920:a,lotka:1920:b},
we assume the first and the second reaction to be autocatalytic,
in particular, we define the kinetics of the reactions as
\[
v_{\cZ\to\cX} = k_{\cZ\to\cX} [\cX], \quad
v_{\cX\to\cY} = k_{\cX\to\cY} [\cX][\cY], \quad
v_{\cY\to\cZ} = k_{\cY\to\cZ} [\cY]
\]
with rate constants $k_{\cZ\to\cX}, k_{\cX\to\cY}, k_{\cY\to\cZ} > 0$.
In terms of chemical reaction network theory (CRNT),
we specify the system as a chemical reaction network with generalized mass-action kinetics,
that is, as a generalized mass-action system~\cite{mueller:regensburger:2012,mueller:regensburger:2014}.
The network arises from a directed graph with edges representing reactions,
\[
\phantom{.} \quad
\begin{array}{c@{}c@{}c}
1 & \longrightarrow & 2 \\
& \nwarrow \quad \swarrow & \\
& 3
\end{array} \quad .
\]
To each node, we assign a (stoichiometric) complex to determine the stoichiometry of the network,
\[
\phantom{.} \quad
\begin{array}{c@{}c@{}c}
\cZ & \longrightarrow & \cX \\
& \nwarrow \quad \swarrow \\
& \cY
\end{array} \quad ,
\]
and additionally a kinetic-order complex to determine the kinetics,
\[
\phantom{.} \qquad\quad
\begin{array}{c@{}c@{}c}
\cX \cdts \cZ & \longrightarrow & \cX \cdts \cX+\cY \\
& \nwarrow \quad \swarrow \\
& \cY
\end{array} \quad .
\]
For the third node, we do not state the kinetic-order complex explicitly
since it coincides with the stoichiometric complex.
The resulting ODE for the concentrations $x=[\cX]$ and $y=[\cY]$ is given by
\begin{align} \label{ode_classical_LV}
\dot x &= k_{12} \, x - k_{23} \, x y, \\
\dot y &= k_{23} \, x y - k_{31} \, y \nonumber
\end{align}
with $k_{12}=k_{\cZ\to\cX}$, $k_{23}=k_{\cX\to\cY}$, $k_{31}=k_{\cY\to\cZ}$.

In this work, we consider more general stoichiometry and kinetics.
In the second reaction,
we allow that $\cX$ and $\cY$ have different stoichiometric coefficients,
in particular, we consider $n \cX \longrightarrow \cY$ with $n>0$.
Further, we allow general power-law kinetics for all reactions,
\[
v_{12} = k_{12} [\cX]^{\al_1} [\cY]^{\be_1}, \quad
v_{23} = k_{23} [\cX]^{\al_2} [\cY]^{\be_2}, \quad
v_{31} = k_{31} [\cX]^{\al_3} [\cY]^{\be_3}
\]
with exponents $\al_1, \be_1, \al_2, \be_2, \al_3, \be_3 \in \R$.
Also this system can be specified as a chemical reaction network with generalized mass-action kinetics,
that is, as a directed graph with stoichiometric and kinetic-order complexes:

\begin{equation} \label{network}
\begin{array}{c@{}c@{}c}
\al_1 \cX + \be_1 \cY \cdts \cZ \hspace{-2ex} & \longrightarrow & \hspace{-2ex} n \cX \cdts \al_2 \cX + \be_2 \cY \\
& \nwarrow \quad \swarrow \\
& \cY & \\[-0.5ex]
& \vdots & \\
& \al_3 \cX + \be_3 \cY &
\end{array}
\end{equation}
The resulting ODE is given by
\begin{align*}
\dot x &= n k_{12} \, x^{\al_1} y^{\be_1} - n k_{23} \, x^{\al_2} y^{\be_2} , \\
\dot y &= k_{23} \, x^{\al_2} y^{\be_2}   - k_{31} \, x^{\al_3} y^{\be_3} .   \nonumber
\end{align*}

In order to define the dynamics on the {\em non-negative} quadrant,
one allows only non-negative exponents.
Further, in order to ensure forward-invariance of the non-negative quadrant, one requires $\al_2,\be_3 > 0$.
However, in this work, we allow real exponents
and consider the dynamics on the {\em positive} quadrant.

In CRNT, one is often interested in results for {\em all} rate constants (and given stoichiometry),
in our example, in existence, uniqueness, and stability of equilibria for all $k_{12}$, $k_{23}$, $k_{31}$ (and given $n$).

Finally, 
we introduce the parameters
\begin{equation} \label{equ_k}
k_1=n k_{12}, \, k_2=n k_{23}, \, k_3=k_{23}, \, k_4=k_{31} ,
\end{equation}
which are in one-to-one correspondence with the stoichiometric coefficient $n$ and the rate constants $k_{12}, \, k_{23}, k_{31}$,
and obtain a generalized Lotka-Volterra scheme~\cite{farkas:noszticzius:1985,dancso:farkas:farkas:szabo:1991},
\begin{align} \label{ode_math}
\dot x &= k_1 \, x^{\al_1} y^{\be_1} - k_2 \, x^{\al_2} y^{\be_2} , \\
\dot y &= k_3 \, x^{\al_2} y^{\be_2} - k_4 \, x^{\al_3} y^{\be_3} . \nonumber
\end{align}

For the technical analysis,
we often consider an ODE which is orbitally equivalent on the positive quadrant and has two exponents less,
\begin{align} \label{ode_trafo}
\dot x &= k_1 \, x^{a_1} y^{b_1} - k_2 , \\
\dot y &= k_3 - k_4 \, x^{a_3} y^{b_3} , \nonumber
\end{align}
where $a_1=\al_1-\al_2$, $b_1=\be_1-\be_2$, $a_3=\al_3-\al_2$, $b_3=\be_3-\be_2$.
Dancs{\'o} et al.~\cite{dancso:farkas:farkas:szabo:1991} studied the ODE~\eqref{ode_math}
in another orbitally equivalent form,
\begin{align} \label{ode_dancso}
\dot x &= k_1 \, x^{\hat p} - k_2 \, x^{p} y^{q} , \\
\dot y &= k_3 \, x^{p} y^{q} - k_4 \, y^{\hat q} , \nonumber
\end{align}
where $\hat p=\al_1-\al_3$, $\hat q=\be_3-\be_1$, $p=\al_2-\al_3$, $q=\be_2-\be_1$.
The authors carried out a stability analysis,
in particular, they studied the Andronov-Hopf bifurcation of the unique positive equilibrium (provided it exists)
and a so-called zip bifurcation.
In this work, we provide a global stability analysis.


\section{Main results} \label{sec:main_results}

We investigate the ODE~\eqref{ode_math},
in particular, we are interested in the qualitative dynamics on the positive quadrant  $\R^2_+$.
We call a positive equilibrium {\em globally asymptotically stable}
if it is Lyapunov stable and from each positive initial conditions the solution converges to the equilibrium.

We introduce
\begin{align} \label{eq:C_mtx}
C =
\begin{pmatrix}
a_1 & b_1 \\
a_3 & b_3
\end{pmatrix}
=
\begin{pmatrix}
\al_1 - \al_2 & \be_1 - \be_2 \\
\al_3 - \al_2 & \be_3 - \be_2
\end{pmatrix}
\quad \text{and} \quad
k=(k_1,k_2,k_3,k_4) \in \R^4_+ 
\end{align}
and start by examining the number of equilibria.

\begin{prop} \label{equilibria}
For the ODE~\eqref{ode_math}, the following statements hold.
\begin{enumerate}[{\rm(i)}]
\item
If $\det C \neq 0$, then there exists a unique positive equilibrium $(x^*,y^*)$ given by
\begin{align} \label{eq:x_star_y_star}
x^* = \left(\frac{k_2}{k_1}\right)^{\textstyle \frac{b_3}{\det C}} \left(\frac{k_3}{k_4}\right)^{\textstyle -\frac{b_1}{\det C}}
\quad \text{and} \quad 
y^* = \left(\frac{k_2}{k_1}\right)^{\textstyle -\frac{a_3}{\det C}}  \left(\frac{k_3}{k_4}\right)^{\textstyle \frac{a_1}{\det C}}.
\end{align}
\item
If $\det C = 0$, then the set of positive equilibria is either empty or infinite, depending on $k$.
\end{enumerate}
\end{prop}

We consider only the generic case (i) and assume $\det C \neq 0$ in the following.
Case (ii), in particular the related zip bifurcation, was studied in detail in~\cite{dancso:farkas:farkas:szabo:1991}.
We proceed by examining the asymptotic stability of the unique positive equilibrium using linearization.

\begin{prop} \label{fix_k_local}
Assume $\det C \neq 0$ and let $(x^*,y^*)$ be the unique positive equilibrium of the ODE~\eqref{ode_math}.
Then the following statements hold.
\begin{enumerate}[{\rm(i)}]
\item
The Jacobian matrix $J$ at $(x^*,y^*)$ is given by
\begin{align*}
J
&= (x^*)^{\al_2} (y^*)^{\be_2} 
\begin{pmatrix}
k_2 & 0 \\
0 & k_3
\end{pmatrix}
\begin{pmatrix}
a_1 & b_1 \\
-a_3 & -b_3
\end{pmatrix}
\begin{pmatrix}
\textstyle \frac{1}{x^*} & 0 \\
0 & \textstyle \frac{1}{y^*}
\end{pmatrix} \\
&= (x^*)^{\al_2} (y^*)^{\be_2}
\begin{pmatrix}
a_1  \frac{k_2}{x^*} &  b_1 \frac{k_2}{y^*} \\
-a_3 \frac{k_3}{x^*} & -b_3 \frac{k_3}{y^*}
\end{pmatrix}
\end{align*}
and hence
\begin{align*}
\sign(\det J) &= - \sign(\det C) , \\
\sign(\tr J)  &= \sign \left( a_1 \frac{k_2}{x^*} - b_3 \frac{k_3}{y^*} \right) .
\end{align*}
\item
The linearization at $(x^*,y^*)$ is asymptotically stable if and only if
$\det C < 0$ and $a_1\frac{k_2}{x^*} - b_3\frac{k_3}{y^*} < 0$.
\end{enumerate}
\end{prop}

We are also interested in asymptotic stability
when the trace of the Jacobian matrix vanishes,
that is, when linearization does not give any information.

\begin{prop} \label{hopf}
For the ODE~\eqref{ode_math}, assume $\det C < 0$ and $a_1\frac{k_2}{x^*} - b_3\frac{k_3}{y^*} = 0$.
Further, let
\[
d_1 = a_3 [(1 + a_3 - a_1) b_1 b_3 - a_1 a_3 (1 + b_1 - b_3)] .
\]
Then the following statements hold.
\begin{enumerate}[{\rm(i)}]
\item
If $d_1 < 0$, the unique positive equilibrium is asymptotically stable.
If $d_1 > 0$, it is repelling.
\item
If we consider a one-parameter family of ODEs~\eqref{ode_math}
along which the eigenvalues of the Jacobian matrix cross the imaginary axis with positive speed,
for example, with parameter $\mu = a_1\frac{k_2}{x^*} - b_3\frac{k_3}{y^*}$,
then an Andronov-Hopf bifurcation occurs at $\mu=0$.
If $d_1 < 0$, the bifurcation is supercritical (and there exists an asymptotically stable closed orbit for small $\mu>0$).
If $d_1 > 0$, it is subcritical (and there exists a repelling closed orbit for small $\mu<0$).
\end{enumerate}
\end{prop}

This result is essentially Theorem 1 in Dancs\'o et al.~\cite{dancso:farkas:farkas:szabo:1991},
apart from the expression $d_1$ for the first focal value.
For details, see Subsection~\ref{sec:strudel}.

If $a_1=b_3=0$ (and hence $d_1=0$) and $a_3b_1 > 0$,
then the unique positive equilibrium is a center, which can be shown by finding a first integral, see~\cite{farkas:noszticzius:1985}.
This includes the classical Lotka-Volterra system \eqref{ode_classical_LV}, where $a_3 = b_1 = -1$.
For further details, see 
Subsection~\ref{subsubsec:first_integral}.

{\bf Remark.}
Llibre~\cite{llibre:2016} claims that Theorem 1 in Dancs\'o et al.~\cite{dancso:farkas:farkas:szabo:1991} is wrong
and provides the following ``counter example'':
$\al_1 = \be_3 = 1$, $\beta_1 = \alpha_3 = 0$, $\al_2 = \be_2 = 2,$ $k_1 = k_2 = 1$, $k_3 = k_4 = 1 + \mu$
and hence
$a_1 = b_3 = -1$, $b_1 = a_3 = -2$, $x^* = y^* = 1$, $\det C < 0$, and $a_1\frac{k_2}{x^*} - b_3\frac{k_3}{y^*} = \mu$.
As a result, one has $d_1 = 0$ (also from the incorrect expression in Theorem 1 in~\cite{dancso:farkas:farkas:szabo:1991}),
and the nature of the Andronov-Hopf bifurcation is not determined by the first focal value.
For $\mu = 0$, there is a first integral $H(x,y) = x+y+ \frac1{xy}$,
see \cite{llibre:2016} or already \cite{dancso:farkas:farkas:szabo:1991} (p.~122, Table I, case (ii) with $\hat p = \hat q = 1$, $p = q = 2$). 
Therefore the Andronov-Hopf bifurcation is degenerate, 
that is, all closed orbits occur at $\mu = 0$.
Llibre's ``counter example'' is an explicit example of a degenerate Andronov-Hopf bifurcation which was already well described by Dancs\'o et al.
Without explaining it, Llibre uses a terminology which is non-standard in the dynamical systems community.
The case of a degenerate Andronov-Hopf bifurcation (where all periodic solutions appear at the critical parameter value) is included in most treatments of this theory,
starting with the original theorems of Andronov and Leontovich~\cite{andronov:leontovich:1937} and Hopf~\cite{hopf:1942}.
See, for example, Lemma 7.2.5 and Theorem 7.2.3 in \cite{farkas:1994}.

%

Next, we characterize the exponents $\al_1$, $\be_1$, $\al_2$, $\be_2$, $\al_3$, $\be_3$ in the ODE~\eqref{ode_math}
for which the unique positive equilibrium is asymptotically stable for all parameters~$k$.
Recall that $k$ depends on the rate constants $k_{12}, k_{23}, k_{31}$ and the stoichiometric coefficient~$n$
of the underlying chemical reaction network~\eqref{network}.
As usual in CRNT, we are also interested in statements which hold for all rate constants, but given stoichiometry,
that is, for all $k$ with $k_2 = n k_3$, see Equations~\eqref{equ_k}.

\begin{thm} \label{all_k_local}
Assume $\det C \neq 0$, that is, for all parameters $k$, the ODE~\eqref{ode_math} admits a unique positive equilibrium.
Further, let $n>0$.
Then the following statements are equivalent.
\begin{enumerate}[{\rm(i)}]
\item For all $k$, the unique positive equilibrium is asymptotically stable.
\item For all $k$ with $k_2 = n k_3$, the unique positive equilibrium is asymptotically stable.
\item $\det C < 0$, $a_1 \le 0 \le b_3$, and $(a_1,b_3) \neq (0,0)$.
\end{enumerate}
\end{thm}

After ruling out periodic solutions, unbounded solutions, and solutions approaching the boundary of the positive quadrant,
we find that -- apart from some boundary cases -- {\em global} asymptotic stability for all $k$ follows from asymptotic stability for all $k$.
\begin{thm} \label{all_k_global}
Assume $\det C \neq 0$, that is, for all parameters $k$, the ODE~\eqref{ode_math} admits a unique positive equilibrium.
Further, let $n>0$.
Then the following statements are equivalent.
\begin{enumerate}[{\rm(i)}]
\item For all $k$, the unique positive equilibrium is globally asymptotically stable.
\item For all $k$ with $k_2 = n k_3$, the unique positive equilibrium is globally asymptotically stable.
\item $\det C < 0$ and either
\begin{enumerate}[$\bullet$]
\item $a_1 < 0 < b_3$,
\item $a_1 < 0 = b_3$, $a_3<0$, and $b_1 \leq -1$, or
\item $a_1 = 0 < b_3$, $a_3 \leq -1$, and $b_1<0$.
\end{enumerate}
\end{enumerate}
\end{thm}

In preparation of our final result, we provide sufficient conditions for precluding global asymptotic stability.
\begin{prop} \label{prop:fwd_inv_sets}
For the ODE~\eqref{ode_math}, assume $\det C < 0$ and one of the following four conditions.
\begin{enumerate}[$\bullet$]
\item $a_1<0$, $b_1<0$, $a_3<0$, $b_3<0$, and $(0< 1 + b_1 - b_3$ or $\det C > a_3+b_3)$.
\item $a_1<0$, $b_1>0$, $a_3>0$, $b_3<0$, and $\det C > a_3+b_3$.
\item $a_1>0$, $b_1>0$, $a_3>0$, $b_3>0$.
\item $a_1>0$, $b_1<0$, $a_3<0$, $b_3>0$, and $(a_1>1$ or $a_1+b_1>0)$.
\end{enumerate}
Then there exists a solution which does not converge to the unique positive equilibrium.
\end{prop}

Finally, we characterize (global) asymptotic stability of the unique positive equilibrium for given $k$.
To reduce the complexity of the problem,
we analyze the special case
$\al_1 = \al$, $\be_1 = 0$, $\al_2 = 1$, $\be_2 = \be$, $\al_3 = 0$, $\be_3 = 1$ with $\al, \be \in \R$.
That is, we consider the ODE
\begin{align}\label{eq:ode_two_exponents_general_k}
\dot x &= k_1 \, x^\al  - k_2 \, x y^\be , \\
\dot y &= k_3 \, x y^\be - k_4 \, y.\nonumber
\end{align}
In terms of the underlying chemical reaction network~\eqref{network},
we study two consecutive autocatalytic reactions with kinetic orders $\al$ and $\be$, respectively,
and a degradation reaction.
To ease the notation, we further assume $k_1 = k_2 = k_3 = k_4 = 1$.

\begin{thm}\label{thm:alpha_beta_system_stability}
For the ODE~\eqref{eq:ode_two_exponents_general_k},
assume $\al \be - \al + 1 \neq 0$ and $k_1 = k_2 = k_3 = k_4 = 1$. 
Then the following statements hold.
\begin{enumerate}[{\rm(i)}]
\item
There exists a unique positive equilibrium, namely $(1,1)$,
and the Jacobian matrix $J$ at $(1,1)$ is given by
\[
J =
\begin{pmatrix}
\al-1 & -\be \\
1     &  \be-1
\end{pmatrix}.
\]
\item
The equilibrium $(1,1)$ is asymptotically stable if and only if
\[
\al \be - \al + 1 > 0, \, \al + \be \leq 2, \, \text{and } (\al,\be) \neq (1,1).
\]
\item
If $\al \be - \al + 1 > 0$, $\al + \be = 2$, and $(\al,\be)\neq(1,1)$,
then the equilibrium $(1,1)$ is asymptotically stable and undergoes a supercritical Andronov-Hopf bifurcation.
\item 
If $(\al,\be)=(1,1)$, then the equilibrium $(1,1)$ is a center,
and all solutions in the positive quadrant are periodic.
\item
The equilibrium $(1,1)$ is globally asymptotically stable if and only if $\al \be - \al + 1 > 0$ and either
\begin{enumerate}[$\bullet$]
\item $\al \leq 1$, $\be \leq 1$, and $(\al,\be)\neq(1,1)$, or
\item $1 < \al \leq \frac{3}{2}$ and $\al-1 \leq \be \leq 2-\al$.
\end{enumerate}
\end{enumerate}
\end{thm}

Figure~\ref{fig:stability_diagram_fixk} illustrates our main results,
in particular, the stability properties of the unique positive equilibrium of the ODE~\eqref{eq:ode_two_exponents_general_k}.
We depict regions in the ($\al$,$\be$)-plane
which result in stable or unstable behavior
either for all parameters $k$ or for $k_1 = k_2 = k_3 = k_4 = 1$.

\begin{figure}[h!]
\begin{center}
\includegraphics[scale=0.78]{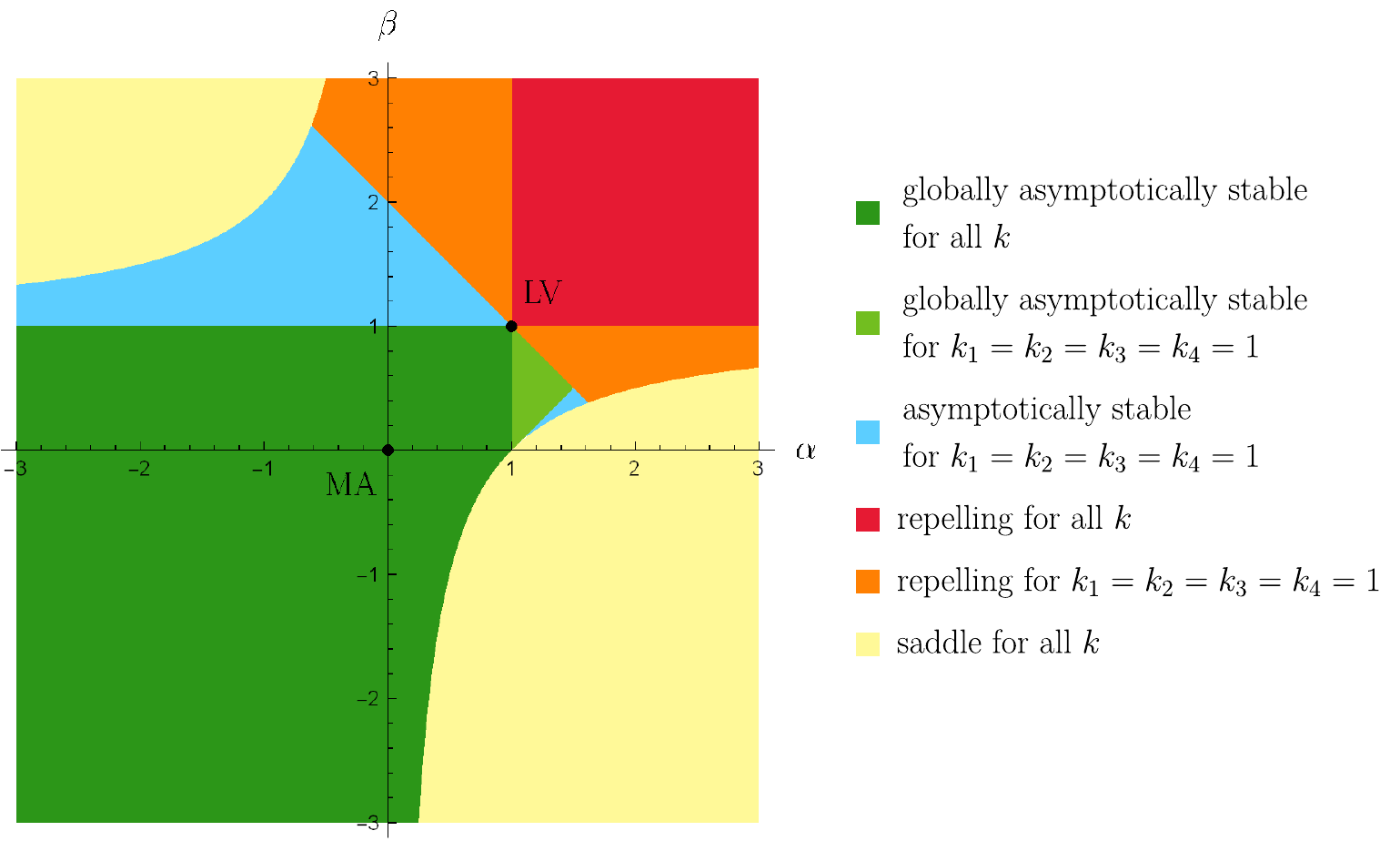}
\end{center}
\caption{Stability of the unique positive equilibrium of the ODE~\eqref{eq:ode_two_exponents_general_k}
depending on the exponents $(\al,\be)$.
The special case $(\al,\be)=(1,1)$ is a classical Lotka-Volterra system,
while $(\al,\be)=(0,0)$ is a linear mass-action system.}
\label{fig:stability_diagram_fixk}
\end{figure}


\section{Proofs} \label{sec:proofs_of_main_results}

It remains to prove the results presented in Section \ref{sec:main_results}.

\subsection{Number of equilibria: Proposition~\ref{equilibria}}

Using the notation \eqref{eq:C_mtx}, a positive equilibrium $(x,y) \in \R^2_+$ of the ODE~\eqref{ode_math} is determined by
\[
x^{a_1}y^{b_1}=\frac{k_2}{k_1} \quad \text{and} \quad x^{a_3}y^{b_3}=\frac{k_3}{k_4}
\]
or, equivalently, by
\[
\begin{pmatrix}
a_1 & b_1 \\
a_3 & b_3
\end{pmatrix}
\begin{pmatrix}
\ln x \\
\ln y
\end{pmatrix}=
\begin{pmatrix}
\ln\frac{k_2}{k_1} \\
\ln\frac{k_3}{k_4}
\end{pmatrix} .
\]
If $\det C \neq 0$, there exists a unique positive equilibrium, and a short calculation shows that it is given by Equations~\eqref{eq:x_star_y_star}.
If $\det C = 0$, the set of equilibria is either infinite or empty
depending on whether the vector $(\ln\frac{k_2}{k_1},\ln\frac{k_3}{k_4})^\mathsf{T}$ is in the image of $C$ or not.
This concludes the proof of Proposition~\ref{equilibria}.


\subsection{Asymptotic stability of the unique positive equilibrium}
\label{subsec:local_stability}

We prove Proposition~\ref{fix_k_local}, compute the first focal value stated in Proposition~\ref{hopf},
discuss the special case $(a_1,b_3)=(0,0)$, and prove Theorem~\ref{all_k_local}.
See the corresponding Subsections \ref{subsubsec:jacobian}, \ref{sec:strudel}, \ref{subsubsec:first_integral}, and \ref{subsubsec:proof_all_k_local}.

\subsubsection{The Jacobian matrix at the unique positive equilibrium: Proposition~\ref{fix_k_local}} \label{subsubsec:jacobian}

Assume $\det C \neq 0$ for the ODE~\eqref{ode_math}.
The Jacobian matrix $J$ at the unique positive equilibrium $(x^*,y^*)$ is given by
\begin{align*}
J_{11} &= k_1 \al_1 (x^*)^{\al_1-1} (y^*)^{\be_1}   - k_2 \al_2 (x^*)^{\al_2-1} (y^*)^{\be_2}, \\
J_{12} &= k_1 \be_1 (x^*)^{\al_1}   (y^*)^{\be_1-1} - k_2 \be_2 (x^*)^{\al_2}   (y^*)^{\be_2-1}, \\
J_{21} &= k_3 \al_2 (x^*)^{\al_2-1} (y^*)^{\be_2}   - k_4 \al_3 (x^*)^{\al_3-1} (y^*)^{\be_3}, \\
J_{22} &= k_3 \be_2 (x^*)^{\al_2}   (y^*)^{\be_2-1} - k_4 \be_3 (x^*)^{\al_3}   (y^*)^{\be_3-1}.
\end{align*}
Using the equilibrium equations $k_1(x^*)^{a_1} (y^*)^{b_1} = k_2$ and $k_3 = k_4 (x^*)^{a_3} (y^*)^{b_3}$, we obtain
\begin{align*}
J=(x^*)^{\al_2} (y^*)^{\be_2}
\begin{pmatrix}
a_1\frac{k_2}{x^*} & b_1\frac{k_2}{y^*} \\
-a_3\frac{k_3}{x^*} & -b_3\frac{k_3}{y^*}
\end{pmatrix}.
\end{align*}
The factorization
\begin{align*}
J=(x^*)^{\al_2} (y^*)^{\be_2} 
\begin{pmatrix}
k_2 & 0 \\
0 & k_3
\end{pmatrix}
\begin{pmatrix}
a_1 & b_1 \\
-a_3 & -b_3
\end{pmatrix}
\begin{pmatrix}
\textstyle \frac{1}{x^*} & 0 \\
0 & \textstyle \frac{1}{y^*}
\end{pmatrix}
\end{align*}
follows directly.
This concludes the proof of Proposition~\ref{fix_k_local} \rm{(i)} which immediately implies Proposition~\ref{fix_k_local} \rm{(ii)}.

\subsubsection{First focal value: Proposition~\ref{hopf}} \label{sec:strudel}

We assume $\det C \neq 0$ and use the unique positive equilibrium $(x^*,y^*)$ to scale the ODE~\eqref{ode_trafo}.
We introduce $K = \frac{k_3}{k_2} \, \frac{x^*}{y^*}$
and obtain the equivalent ODE
\begin{align*}
\dot x &= f(x,y) = x^{a_1} y^{b_1} - 1 , \\
\dot y &= g(x,y) = K \left( 1 - x^{a_3} y^{b_3} \right) , \nonumber
\end{align*}
which has the unique positive equilibrium $(1,1)$.
If the equilibrium undergoes an Andronov-Hopf bifurcation,
then the determinant of the Jacobian matrix is positive (that is, $\det C < 0$)
and its trace is zero (that is, $a_1=K\,b_3$).
In order to determine the sign of the first focal value $D_1$,
we use Equation~\eqref{strudel} from Appendix A.
In particular, we compute the partial derivatives of $f$ and $g$ up to order three at $(1,1)$.
The resulting first focal value $D_1$ has the same sign as the expression
\[
d_1 = a_3 [(1 + a_3 - a_1) b_1 b_3 - a_1 a_3 (1 + b_1 - b_3)] .
\]

In~\cite{dancso:farkas:farkas:szabo:1991},
the authors study the ODE~\eqref{ode_math} in the form \eqref{ode_dancso}.
In particular, they compute the first focal value.
However, Equation~(22) in~\cite{dancso:farkas:farkas:szabo:1991} is not correct.
In their notation, the correct formula is given by
\[
G = p[p(\hat p-p)(\hat q-1) - q(\hat q-q)(\hat p-1)] .
\]

We also correct the three lemmas in the Appendix of~\cite{dancso:farkas:farkas:szabo:1991}.
\begin{lem}[Lemma 1 in \cite{dancso:farkas:farkas:szabo:1991}]
Consider the system
\[
\dot z_1 = A (z_1^{a_1}-z_2^{a_2}), \quad \dot z_2 = B (z_1^{b_1}-z_2^{b_2}),
\]
where $\det J = AB(-a_1b_2+b_1a_2)>0$ and $\tr J = Aa_1-Bb_2 = 0$.
At the equilibrium $(1,1)$,
the first focal value has the same sign as
\[
g = A b_1 [(1 + b_1 - a_1) a_2 b_2 - a_1 b_1 (1 + a_2 - b_2)] .
\]
\end{lem}

\begin{lem}[Lemma 2 in \cite{dancso:farkas:farkas:szabo:1991}]
Consider the system
\[
\dot z_1 = A (\e^{a_1z_1}-z_2^{a_2}), \quad \dot z_2 = B (e^{b_1z_1}-z_2^{b_2}),
\]
where $\det J = AB(-a_1b_2+b_1a_2)>0$ and $\tr J = Aa_1-Bb_2 = 0$.
At the equilibrium $(0,1)$,
the first focal value has the same sign as
\[
g = -A a_1 (1 + a_2 - b_2) .
\]
\end{lem}

\begin{lem}[Lemma 3 in \cite{dancso:farkas:farkas:szabo:1991}]
Consider the system
\[
\dot z_1 = A (\e^{a_1z_1}-\e^{a_2z_2}), \quad \dot z_2 = B (e^{b_1z_1}-\e^{b_2z_2}),
\]
where $\det J = AB(-a_1b_2+b_1a_2)>0$ and $\tr J = Aa_1-Bb_2 = 0$.
At the equilibrium $(0,0)$,
the first focal value is zero.
\end{lem}

We note that an expression for the first focal value has been derived for planar S-systems, see~\cite{yin:voit:2008}.

\subsubsection{The special case $(a_1,b_3)=(0,0)$}
\label{subsubsec:first_integral}

We describe the behavior of the ODE~\eqref{ode_math} for $(a_1,b_3)=(0,0)$. 
In particular, we find a first integral, see also~\cite{farkas:noszticzius:1985}.

In the special case $(a_1,b_3)=(0,0)$, the ODE \eqref{ode_math} takes the form
\begin{align} \label{ode_a1_0_b3_0}
\dot x &= x^{\al_2}y^{\be_2} (k_1 \, y^{b_1} - k_2) , \\
\dot y &= x^{\al_2}y^{\be_2} (k_3 \, - k_4 \, x^{a_3}) . \nonumber
\end{align}

If $a_3b_1<0$, then $\det C > 0$ and the unique positive equilibrium is a saddle.
If $a_3b_1>0$, we show that the unique positive equilibrium $(x^*,y^*)$ is a center.
Let $V \colon \R^2_+ \to \R$ be a continuously differentiable function with 
\begin{align} \label{lyapunov}
(\partial_1V)(x,y) &= k_3 \bigg( 1 - \Big(\frac{x}{x^*}\Big)^{a_3} \bigg),\\
(\partial_2V)(x,y) &= k_2 \bigg( 1 - \Big(\frac{y}{y^*}\Big)^{b_1} \bigg). \nonumber
\end{align}
A short calculation shows that $V$ is a first integral for \eqref{ode_a1_0_b3_0}.
Since level sets of $V$ that are close enough to $(x^*,y^*)$ are closed curves, we obtain that $(x^*,y^*)$ is indeed a center.

We can also deduce the global behavior for $(a_1,b_3)=(0,0)$ and $a_3b_1 > 0$.
Recall that we defined both the ODE~\eqref{ode_math} and the function $V$ on the positive quadrant.
If $a_3 \leq -1$ and $b_1 \leq -1$ ,
then $V(x,y)$ approaches infinity whenever $x$ or $y$ approaches zero or infinity.
Hence, all level sets of $V$ are closed curves,
see Figure~\ref{fig:streamplot_center_stay_away_from_the_boundary} for the corresponding phase portrait.
On the other hand, if at least one of $a_3 > -1$ and $b_1 > -1$ holds,
then some level sets intersect the boundary of the positive quadrant.
Hence, the level sets of $V$ that are close enough to $(x^*,y^*)$ are still closed curves,
but some level sets connect two points on the boundary (or touch the boundary in one point).
In Figure~\ref{fig:streamplots_center_meet_the_boundary},
we have collected the phase portraits for four cases where at least one of $a_3 \le -1$ or $b_1 \le -1$ is violated.

Solutions approaching the boundary forward (respectively, backward) in time either
\begin{enumerate}[$\bullet$]
\item reach the boundary in finite time, and thus these solutions are not defined for all positive (respectively, negative) times or
\item  stay in the positive quadrant for all positive (respectively, negative) times.
\end{enumerate}
Which of the two cases occurs is not determined solely by the values of $a_3$ and $b_1$, but also by the values of $\al_2$ and $\be_2$.

\subsubsection{Asymptotic stability for all $k$: Theorem~\ref{all_k_local}}
\label{subsubsec:proof_all_k_local}

Assume $\det C \neq 0$ and let $J$ be the Jacobian matrix at the unique positive equilibrium $(x^*,y^*)$.
Then, by Proposition~\ref{fix_k_local}~\rm(i), we have $\tr J \leq 0$ if and only if $a_1 - K b_3\leq 0$, where $K = \frac{k_3}{k_2}\frac{x^*}{y^*}$.
We claim that $K$ can take any positive value if either
\begin{enumerate}[$\bullet$]
\item the parameters $k_1$, $k_2$, $k_3$, $k_4$ can take any positive values, or
\item the parameters $k_1$, $k_2$, $k_3$, $k_4$ can take any positive values with $k_2=nk_3$.
\end{enumerate}
To prove the claim, let $k_2$, $k_3$, $x^*$, and $y^*$ be arbitrary positive numbers (with $k_2=nk_3$ in the second case).
Then the unique positive equilibrium of the ODE~\eqref{ode_math} with $k_1=k_2 \, (x^*)^{-a_1}(y^*)^{-b_2}$ and $k_4=k_3 \, (x^*)^{-a_3}(y^*)^{-b_3}$ is indeed $(x^*,y^*)$.
This concludes the proof of the claim.

As a consequence, each of the statements \rm(i) and \rm(ii) in Theorem~\ref{all_k_local} implies $a_1 \leq 0 \leq b_3$.
By Proposition~\ref{fix_k_local}~\rm(i), asymptotic stability of the unique positive equilibrium implies $\det C < 0$.
By the discussion in Subsection \ref{subsubsec:first_integral}, the unique positive equilibrium is not asymptotically stable if $(a_1,b_3)=(0,0)$.
Thus, each of the statements \rm(i) and \rm(ii) in Theorem~\ref{all_k_local} implies \rm(iii).

On the other hand,
statement \rm(iii) in Theorem~\ref{all_k_local} implies $\tr J < 0$ and $\det J > 0$.
Therefore, both eigenvalues of $J$ have negative real part, and therefore the positive equilibrium is asymptotically stable. 
Thus, statement \rm(iii) implies both \rm(i) and \rm(ii).
This concludes the proof of Theorem~\ref{all_k_local}.

\subsection{Global asymptotic stability for all $k$: Theorem~\ref{all_k_global}}

By Theorem~\ref{all_k_local}, the unique positive equilibrium of the ODE~\eqref{ode_math} is asymptotically stable for all $k$
if and only if $\det C < 0$, $a_1 \leq 0 \leq b_3$, and $(a_1,b_3)\neq(0,0)$.
Theorem~\ref{all_k_global} states that -- apart from some boundary cases -- asymptotic stability implies global asymptotic stability.
To prove it, we need to rule out periodic solutions, unbounded solutions, and solutions approaching the boundary of the positive quadrant.
For the technical analysis, we consider the ODE~\eqref{ode_trafo}.

\subsubsection{Precluding periodic solutions}
\label{subsubsec:preclude_periodic}

To preclude periodic solutions, we use the Bendixson-Dulac test.
Let $h \colon \R^2_+ \to \R_+$ and $(f,g) \colon \R^2_+ \to \R^2$ be two continuously differentiable functions. Then
\begin{align*}
\frac{\dive(hf,hg)}{h} = \dive(f,g) + f\frac{\partial_1 h}{h} + g\frac{\partial_2 h}{h}
= \partial_1 f + \partial_2 g + f\frac{\partial_1 h}{h} + g\frac{\partial_2 h}{h}.
\end{align*}

Let $f(x,y) = k_1 x^{a_1} y^{b_1} - k_2$ and $g(x,y) = k_3 - k_4 x^{a_3} y^{b_3}$ denote the right-hand side of the ODE~\eqref{ode_trafo}
and choose $h(x,y) = x^{-p} y^{-q}$ as Dulac function.
Then $\frac{\partial_1 h}{h} = -\frac{p}{x}$, $\frac{\partial_2 h}{h} = -\frac{q}{y}$, and
\[
\frac{\dive(hf,hg)}{h} = 
k_1 (a_1 - p) x^{a_1 - 1} y^{b_1} +
k_2 p x^{-1}                      +
k_3 (-q) y^{-1}                   +
k_4 (q - b_3) x^{a_3} y^{b_3 - 1}.
\]
Thus, if $a_1 \leq 0 \leq b_3$ and $(a_1,b_3)\neq(0,0)$,
and $p$ and $q$ are chosen such that $a_1 \leq p \leq 0 \leq q \leq b_3$, then $\frac{\dive(hf,hg)}{h} < 0$ on $\R^2_+$.
Therefore, by the Bendixson-Dulac test, the ODE~\eqref{ode_trafo}
does not admit periodic solutions.

\subsubsection{Precluding unbounded solutions} \label{subsubsec:bounded_orbits}

In order to prove boundedness of the solutions in the case $a_1 < 0 < b_3$ and $\det C < 0$,
we consider all possible signs of $a_3$ and $b_1$ and the corresponding nullcline geometries.
For phase portraits in the nine cases, see Figure~\ref{fig:streamplots_very_stable_region}. 
In two cases (top left and bottom right),
solutions may spiral around the unique positive equilibrium. 
Since the divergence of the right-hand side of the ODE~\eqref{ode_trafo} is negative, 
see Subsection~\ref{subsubsec:preclude_periodic}, they can spiral inwards only (anti-clockwise and clockwise, respectively).
In the other seven cases, solutions are ultimately monotonic.

In the cases $a_1 < 0 = b_3$ and $a_1 = 0 < b_3$,
we prove boundedness of the solutions by providing a Lyapunov function.
In fact, we use a continuously differentiable function $V \colon \R^2_+ \to \R$ defined by Equations~\eqref{lyapunov}. 
Let $f(x,y) = k_1 x^{a_1} y^{b_1} - k_2$ and $g(x,y) = k_3 - k_4 x^{a_3} y^{b_3}$ denote the right-hand side of the ODE~\eqref{ode_trafo}.
The time derivative of $V$ along a solution amounts to
\begin{align*}
\dot V &= (\partial_1 V) \, \dot x + (\partial_2 V) \, \dot y \\
&= (\partial_1 V) \, f + (\partial_2 V) \, g \\
&= x^{a_3}y^{b_1} \cdot
\begin{cases}
k_1k_3 [x^{-a_3}-(x^*)^{-a_3}][x^{a_1}-(x^*)^{a_1}], & \text{if } b_3=0, \\
(-k_2k_4) [y^{-b_1}-(y^*)^{-b_1}][y^{b_3}-(y^*)^{b_3}], & \text{if } a_1 = 0,
\end{cases}
\end{align*}
where we use $a_3 b_1 > 0$ (which follows from $\det C < 0$).
%
We first consider $a_3,b_1 < 0$.
In both cases, $a_1 < 0 = b_3$ and $a_1 = 0 < b_3$,
we obtain that $\dot V \le 0$ in~$\R^2_+$.
Since the sublevel sets of $V$ are bounded subsets of $\mathbb{R}^2_+$,
the solutions stay bounded forward in time.
It remains to consider $a_3, b_1 > 0$.
Here, we obtain that $\dot V \ge 0$, and the superlevel sets of $V$ are bounded.
This concludes the proof of the boundedness of the solutions of the ODE~\eqref{ode_trafo}.

\subsubsection{Solutions approaching the boundary of the positive quadrant}
\label{subsubsec:approach_the_boundary_in_the_boundary_case}

We assume $a_1 \leq 0 \leq b_3$ and $\det C < 0$.
We prove that no solution of the ODE~\eqref{ode_trafo} approaches the boundary of the positive quadrant if and only if either
\begin{enumerate}[$\bullet$]
\item $a_1 < 0 < b_3$,
\item $a_1 < 0 = b_3$, $a_3 < 0$, and $b_1 \leq -1$, or
\item $a_1 = 0 < b_3$, $a_3 \leq -1$, and $b_1 < 0$.
\end{enumerate}

In the case $a_1 < 0 < b_3$,
this follows from the nullcline geometries shown in Figure~\ref{fig:streamplots_very_stable_region}.
See the discussion at the beginning of Subsection~\ref{subsubsec:bounded_orbits}.

In the cases $a_1 < 0 = b_3$ and $a_1 = 0 < b_3$, $\det C < 0$ implies $a_3 b_1 > 0$.
We first consider $a_3,b_1 > 0$ and show that there exist solutions approaching the boundary.

In the case $a_1 < 0 = b_3$ (and $a_3 , b_1 > 0$),
the vector field is defined even for $y=0$ and $x>0$ and points transversally out of the non-negative quadrant for $y=0$ and $x > x^*$.
For the phase portrait, see the top left panel in Figure~\ref{fig:streamplots_boundary_of_the_very_stable_region}.

In the case $a_1 = 0 < b_3$ (and $a_3 , b_1 > 0$),
the vector field is defined even for $x=0$ and $y>0$ and points transversally out of the non-negative quadrant for $x=0$ and $0 < y < y^*$.
See the top right panel in Figure~\ref{fig:streamplots_boundary_of_the_very_stable_region}.

It remains to examine the cases
\begin{enumerate}[$\bullet$]
\item $a_1 < 0 = b_3$ and $a_3, b_1 < 0$ and
\item $a_1 = 0 < b_3$ and $a_3, b_1 < 0$.
\end{enumerate}

In the case $a_1 < 0 = b_3$ (and $a_3 ,b_1 < 0$), the $y$-axis is repelling.
If additionally $b_1 \leq -1$ then the level sets of the Lyapunov function $V$ defined by Equations~\eqref{lyapunov} 
are disjoint from the $x$-axis, and therefore no solution can approach the $x$-axis either.
For the phase portrait, see the middle left panel in Figure~\ref{fig:streamplots_boundary_of_the_very_stable_region}.
If alternatively $-1 < b_1 < 0$ then we show that some of the solutions approach the $x$-axis. Consider the auxiliary ODE
\begin{align} \label{ode_aux_b3_zero}
\dot x &= k_1 x^{a_1} y^{b_1} , \\
\dot y &= k_3 - k_4 \, x^{a_3} \nonumber
\end{align}
in the positive quadrant.
Since the ODE~\eqref{ode_aux_b3_zero} is separable, one can solve it explicitly.
We are particularly interested in a solution that approaches the boundary of the positive quadrant at the point $(x^*,0)$.
The orbit of such a solution is given by the curve
\begin{align*}
y = \left[\frac{(b_1+1)}{k_1}\left(
k_3\frac{x^{1-a_1}-(x^*)^{1-a_1}}{1-a_1} - 
k_4\frac{x^{1+a_3-a_1}-(x^*)^{1+a_3-a_1}}{1+a_3-a_1}
\right)\right]^\frac{1}{b_1+1}
\end{align*}
for $0 < x < x^*$ if $1+a_3-a_1 \neq 0$ and
\begin{align*}
y = \left[\frac{(b_1+1)}{k_1}\left(
k_3\frac{x^{1-a_1}-(x^*)^{1-a_1}}{1-a_1} - 
k_4\left(\ln x - \ln x^*\right)
\right)\right]^\frac{1}{b_1+1}
\end{align*}
for $0 < x < x^*$ if $1+a_3-a_1 = 0$.
In any case,
the unique positive equilibrium $(x^*,y^*)$ of the ODE~\eqref{ode_trafo} lies to the right of this curve,
and solutions of the ODE~\eqref{ode_trafo} that start on or to the left of the curve approach the $x$-axis at a point $(\tilde x,0)$ with $0<\tilde x\le x^*$.
For the phase portrait,
see the bottom left panel in Figure~\ref{fig:streamplots_boundary_of_the_very_stable_region},
and for the construction of the curve and the resulting forward invariant set of the ODE~\eqref{ode_trafo},
see the left panel in Figure~\ref{fig:fwd_invar_sets_boundary_of_the_very_stable_region}.

In case $a_1 = 0 < b_3$ (and $a_3 ,b_1 < 0$), the $x$-axis is repelling.
If additionally $a_3 \leq -1$, 
then the level sets of the Lyapunov function $V$ are disjoint from the $y$-axis, and therefore no solution can approach the $y$-axis either.
For the phase portrait, see the middle right panel in Figure~\ref{fig:streamplots_boundary_of_the_very_stable_region}.
If alternatively $-1 < a_3 < 0$ then we show that some of the solutions approach the $y$-axis. Consider the auxiliary ODE
\begin{align} \label{ode_aux_a1_zero}
\dot x &= k_1 y^{b_1} - k_2, \\
\dot y &= - k_4 \, x^{a_3}y^{b_3} \nonumber
\end{align}
in the positive quadrant.
Since the ODE~\eqref{ode_aux_a1_zero} is separable, one can solve it explicitly.
We are particularly interested in a solution that approaches the boundary of the positive quadrant at the point $(0,y^*)$.
The orbit of such a solution is given by the curve
\begin{align*}
x = \left[-\frac{(a_3+1)}{k_4}\left(
k_1\frac{y^{1+b_1-b_3}-(y^*)^{1+b_1-b_3}}{1+b_1-b_3} - 
k_2\frac{y^{1-b_3}-(y^*)^{1-b_3}}{1-b_3}
\right)\right]^\frac{1}{a_3+1}
\end{align*}
for $y > y^*$ if $1+b_1-b_3 \neq 0$ and $b_3 \neq 1$,
\begin{align*}
x = \left[-\frac{(a_3+1)}{k_4}\left(
k_1(\ln y - \ln y^*) - 
k_2\frac{y^{1-b_3}-(y^*)^{1-b_3}}{1-b_3}
\right)\right]^\frac{1}{a_3+1}
\end{align*}
for $y > y^*$ if $1+b_1-b_3 = 0$, and
\begin{align*}
x = \left[-\frac{(a_3+1)}{k_4}\left(
k_1\frac{y^{b_1}-(y^*)^{b_1}}{b_1} - 
k_2(\ln y - \ln y^*)
\right)\right]^\frac{1}{a_3+1}
\end{align*}
for $y > y^*$ if $b_3 = 1$.
In any case,
the unique positive equilibrium $(x^*,y^*)$ of the ODE~\eqref{ode_trafo} lies below this curve,
and solutions of the ODE~\eqref{ode_trafo} that start on or above the curve approach the $y$-axis at a point $(0,\tilde y)$ with $\tilde y \ge y^*$.
For the phase portrait,
see the bottom right panel in Figure~\ref{fig:streamplots_boundary_of_the_very_stable_region},
and for the construction of the curve and the resulting forward invariant set of the ODE~\eqref{ode_trafo},
see the right panel in Figure~\ref{fig:fwd_invar_sets_boundary_of_the_very_stable_region}.

This concludes the proof of Theorem~\ref{all_k_global}.


\subsection{Preclusion of global stability: Proposition~\ref{prop:fwd_inv_sets}}

We assume $\det C < 0$ and $(a_1,b_3)\neq(0,0)$. 
The unique positive equilibrium $(x^*,y^*)$ is asymptotically stable (respectively, repelling) if $a_1 \leq 0 \leq b_3$ (respectively, if $a_1 \geq 0 \geq b_3$),
since the trace of the Jacobian matrix at $(x^*,y^*)$ is negative (respectively, positive), see Proposition~\ref{fix_k_local}.

Proposition~\ref{prop:fwd_inv_sets} deals with the cases $a_1,b_3<0$ and $a_1,b_3>0$. 
In both cases, $\det C < 0$ implies $a_3 b_1 >0$.
In each of the four subcases in Proposition~\ref{prop:fwd_inv_sets},
we construct a closed forward invariant set
that does not contain the unique positive equilibrium, see Figure~\ref{fig:fwd_invar_sets}. 

We prove the four statements in Proposition~\ref{prop:fwd_inv_sets} separately,
see Lemmas \ref{lemma:a1_neg_b1_neg_a3_neg_b3_neg}, \ref{lemma:a1_neg_b1_pos_a3_pos_b3_neg},
\ref{lemma:a1_pos_b1_pos_a3_pos_b3_pos}, and \ref{lemma:a1_pos_b1_neg_a3_neg_b3_pos} below. 
In order to ease the notation in the proofs of these lemmas, we consider the ODE~\eqref{ode_trafo} with $k_1=k_2=k_3=k_4=1$, that is,
\begin{align} \label{eq:ode_abcd_exponents_equal_rate_constants}
\dot x & = x^{a_1} y^{b_1} - 1, \\
\dot y & = 1 - x^{a_3} y^{b_3}. \nonumber
\end{align}
However, the assumption $k_1=k_2=k_3=k_4=1$ is not necessary for the validity of the lemmas.

\begin{lem} \label{lemma:a1_neg_b1_neg_a3_neg_b3_neg}
Assume $a_1<0$, $b_1<0$, $a_3<0$, $b_3<0$, and $\det C < 0$. If
\begin{align*}
1 + b_1 - b_3>0 ~ \text{ or } ~ \det C > a_3 + b_3 ,
\end{align*}
then there exist $\gamma < -\frac{a_3}{b_3}$ and $x_0 > 1$ such that the set 
\begin{align*}
\{(x,y)\in\R^2_+ ~|~ x \geq x_0 ~ \text{and} ~ 0 < y \leq x^\gamma\}
\end{align*}
is forward invariant under the ODE~\eqref{eq:ode_abcd_exponents_equal_rate_constants}. (See the top left panel in Figure~\ref{fig:fwd_invar_sets}.)
\end{lem}
\begin{proof}
We want to find $\gamma < -\frac{a_3}{b_3}$ and $x_0 > 1$ such that, along the curve $y=x^\gamma$,
\begin{align*}
\dd{y}{x} = \frac{\dot y}{\dot x} =
\frac{1 - x^{a_3}(x^\gamma)^{b_3}}{{x^{a_1}(x^\gamma)^{b_1} - 1}} < \gamma x^{\gamma-1} =
\dd{x^\gamma}{x} \quad \text{for all } x \geq x_0
\end{align*}
or, equivalently,
\begin{align*}
\gamma^{-1}x^{1-\gamma}\frac{1 - x^{a_3+\gamma b_3}}{{x^{a_1+\gamma b_1} - 1}} > 1  \quad \text{for all } x \geq x_0.
\end{align*}
The assumptions imply $-\frac{a_3}{b_3} < -\frac{a_1}{b_1} < 0$,
and $\gamma < -\frac{a_3}{b_3}$ further implies $a_3+\gamma b_3>0$ and $a_1+\gamma b_1>0$.
Hence,
\begin{align*}
\lim_{x\to\infty}\gamma^{-1}x^{1-\gamma}\frac{1 - x^{a_3+\gamma b_3}}{{x^{a_1+\gamma b_1} - 1}} = 
\lim_{x\to\infty}-\gamma^{-1}x^{(1-a_1+a_3)-\gamma(1+b_1-b_3)} = \infty,
\end{align*}
if additionally 
\begin{align*}
(1-a_1+a_3)-\gamma(1+b_1-b_3) > 0 .
\end{align*}
If $1+b_1-b_3 > 0$, we just choose $-\gamma$ large enough.
If $\det C > a_3 + b_3$, a short calculation shows $(1-a_1+a_3)+\frac{a_3}{b_3}(1+b_1-b_3)>0$.
Hence, if $1+b_1-b_3 = 0$, the above inequality holds (for every $\gamma$),
and if $1+b_1-b_3 < 0$, we obtain
\[
\frac{1-a_1+a_3}{1+b_1-b_3} < -\frac{a_3}{b_3}
\]
and choose $\gamma$ between the left- and right-hand side of the last inequality. \qed
\end{proof}

\begin{lem}\label{lemma:a1_neg_b1_pos_a3_pos_b3_neg}
Assume $a_1<0$, $b_1>0$, $a_3>0$, $b_3<0$, and $\det C < 0$. If
\begin{align*}
\det C > a_3+b_3 ,
\end{align*}
then there exist $\gamma > -\frac{a_3}{b_3}$ and $x_0 > 1$ such that the set 
\begin{align*}
\{(x,y)\in\R^2_+ ~|~ x \geq x_0 ~ \text{and} ~ x^\gamma \leq y\}
\end{align*}
is forward invariant under the ODE~\eqref{eq:ode_abcd_exponents_equal_rate_constants}. (See the top right panel in Figure~\ref{fig:fwd_invar_sets}.)
\end{lem}
\begin{proof}
We want to find $\gamma > -\frac{a_3}{b_3}$ and $x_0 > 1$ such that, along the curve $y=x^\gamma$,
\begin{align*}
\dd{y}{x} = \frac{\dot y}{\dot x} =
\frac{1 - x^{a_3}(x^\gamma)^{b_3}}{{x^{a_1}(x^\gamma)^{b_1} - 1}} > \gamma x^{\gamma-1} =
\dd{x^\gamma}{x} \quad \text{for all } x \geq x_0
\end{align*}
or, equivalently,
\begin{align*}
\gamma^{-1}x^{1-\gamma}\frac{1 - x^{a_3+\gamma b_3}}{{x^{a_1+\gamma b_1} - 1}} > 1 \quad \text{for all } x \geq x_0.
\end{align*}
The assumptions imply $0 < -\frac{a_1}{b_1} < -\frac{a_3}{b_3}$,
and $\gamma > -\frac{a_3}{b_3}$ further implies $a_3+\gamma b_3<0$ and $a_1+\gamma b_1>0$.
Hence,
\begin{align*}
\lim_{x\to\infty}\gamma^{-1}x^{1-\gamma}\frac{1 - x^{a_3+\gamma b_3}}{{x^{a_1+\gamma b_1} - 1}} = 
\lim_{x\to\infty}\gamma^{-1}x^{(1-a_1)-\gamma(1+b_1)} = \infty,
\end{align*}
if additionally 
\begin{align*}
(1-a_1)-\gamma(1+b_1) > 0.
\end{align*}
If $\det C > a_3 + b_3$, then $(1-a_1)+\frac{a_3}{b_3}(1+b_1)>0$.
Hence, we obtain
\[
-\frac{a_3}{b_3} < \frac{1-a_1}{1+b_1}
\]
and choose $\gamma$ between the left- and right-hand side of the last inequality. \qed
\end{proof}

\begin{lem}\label{lemma:a1_pos_b1_pos_a3_pos_b3_pos}
Assume $a_1>0$, $b_1>0$, $a_3>0$, $b_3>0$, and $\det C < 0$. Then for all $-\frac{a_1}{b_1} < \gamma < 0$ there exists $0 < x_0 < 1$ such that the set 
\begin{align*}
\{(x,y)\in\R^2_+ ~|~ 0 < x \leq x_0 ~ \text{and} ~ x_0^\gamma \leq y \leq x^\gamma\}
\end{align*}
is forward invariant under the ODE~\eqref{eq:ode_abcd_exponents_equal_rate_constants}. (See the bottom left panel in Figure~\ref{fig:fwd_invar_sets}.)
\end{lem}
\begin{proof}
Fix $-\frac{a_1}{b_1} < \gamma < 0$.
We claim that there exists $0 < x_0 < 1$ such that, along the curve $y=x^\gamma$,
\begin{align*}
\dd{y}{x} = \frac{\dot y}{\dot x} =
\frac{1 - x^{a_3}(x^\gamma)^{b_3}}{{x^{a_1}(x^\gamma)^{b_1} - 1}} > \gamma x^{\gamma-1} =
\dd{x^\gamma}{x} \quad \text{for all } x \geq x_0
\end{align*}
or, equivalently,
\begin{align*}
\gamma^{-1}x^{1-\gamma}\frac{1 - x^{a_3+\gamma b_3}}{{x^{a_1+\gamma b_1} - 1}} < 1 \quad \text{for all}~ 0 < x \leq x_0.
\end{align*}
The assumptions imply $-\frac{a_3}{b_3} < -\frac{a_1}{b_1} < 0$,
and $-\frac{a_1}{b_1} < \gamma < 0$ further implies $a_3+\gamma b_3>0$ and $a_1+\gamma b_1>0$.
Hence,
\begin{align*}
\lim_{x \to 0}\gamma^{-1}x^{1-\gamma}\frac{1 - x^{a_3+\gamma b_3}}{{x^{a_1+\gamma b_1} - 1}} = 
\lim_{x\to 0}-\gamma^{-1}x^{1-\gamma} = 0 .
\end{align*}
\qed
\end{proof}

\begin{lem}\label{lemma:a1_pos_b1_neg_a3_neg_b3_pos}
Assume $a_1>0$, $b_1<0$, $a_3<0$, $b_3>0$, and $\det C < 0$. If
\begin{align*}
1-a_1 < 0 ~ \text{ or } ~ a_1 + b_1 > 0 ,
\end{align*}
then there exist $0 < \gamma < -\frac{a_1}{b_1}$ and $x_0 > 1$ such that the set 
\begin{align*}
\{(x,y)\in\R^2_+ ~|~ x \geq x_0 ~ \text{and} ~ x_0^\gamma \leq y \leq x^\gamma\}
\end{align*}
is forward invariant under the ODE~\eqref{eq:ode_abcd_exponents_equal_rate_constants}. (See the bottom right panel in Figure~\ref{fig:fwd_invar_sets}.)
\end{lem}
\begin{proof}
We want to find $0 < \gamma < -\frac{a_1}{b_1}$ and $x_0 > 1$ such that, along the curve $y=x^\gamma$,
\begin{align*}
\dd{y}{x} = \frac{\dot y}{\dot x} =
\frac{1 - x^{a_3}(x^\gamma)^{b_3}}{{x^{a_1}(x^\gamma)^{b_1} - 1}} < \gamma x^{\gamma-1}
\dd{x^\gamma}{x} \quad \text{for all } x \geq x_0
\end{align*}
or, equivalently,
\begin{align*}
\gamma^{-1}x^{1-\gamma}\frac{1 - x^{a_3+\gamma b_3}}{{x^{a_1+\gamma b_1} - 1}} < 1 \quad \text{for all}~ x \geq x_0.
\end{align*}
The assumptions imply $0 < -\frac{a_1}{b_1} < -\frac{a_3}{b_3}$,
and $0 < \gamma < -\frac{a_1}{b_1}$ further implies $a_3+\gamma b_3<0$ and $a_1+\gamma b_1>0$.
Hence,
\begin{align*}
\lim_{x\to\infty}\gamma^{-1}x^{1-\gamma}\frac{1 - x^{a_3+\gamma b_3}}{{x^{a_1+\gamma b_1} - 1}} = 
\lim_{x\to\infty}\gamma^{-1}x^{(1-a_1)-\gamma(1+b_1)} = 0,
\end{align*}
if additionally
\begin{align*}
(1-a_1)-\gamma(1+b_1) < 0. 
\end{align*}
If $1-a_1 < 0$, we just choose $\gamma$ small enough.
If $a_1 + b_1 > 0$, a short calculation shows $(1-a_1)+\frac{a_1}{b_1}(1+b_1) < 0$.
Hence, if $1+b_1 = 0$, the above inequality holds (for every $\gamma$),
and if $1+b_1 > 0$, we obtain
\[
\frac{1-a_1}{1+b_1} < -\frac{a_1}{b_1}
\]
and choose $\gamma$ between the left- and right-hand side of the last inequality. \qed
\end{proof}


\subsection{Global asymptotic stability for a particular $k$: Theorem~\ref{thm:alpha_beta_system_stability}}

We consider the ODE
\begin{align} \label{eq:ode_two_exponents}
\dot x &= x^\al - x y^\be , \\
\dot y &= x y^\be - y \nonumber
\end{align}
and assume $\al\be - \al + 1 \neq 0$.

\subsubsection{Local behavior: Theorem~\ref{thm:alpha_beta_system_stability}~\rm(i), \rm(ii), \rm(iii), and \rm(iv)}

Most of the statements in Theorem~\ref{thm:alpha_beta_system_stability} are direct consequences of the results for the ODE~\eqref{ode_math}
with $\al_1 = \al$, $\be_1 = 0$, $\al_2 = 1$, $\be_2 = \be$, $\al_3 = 0$, $\be_3 = 1$,
and hence
\begin{align*}
C =
\begin{pmatrix}
a_1 & b_1 \\
a_3 & b_3
\end{pmatrix}
=
\begin{pmatrix}
\al - 1 &    - \be \\
     - 1 & 1 - \be
\end{pmatrix} ,
\end{align*}
and $k_1=k_2=k_3=k_4=1$.

Theorem~\ref{thm:alpha_beta_system_stability} \rm{(i)} follows from Proposition~\ref{equilibria} \rm{(i)} and Proposition~\ref{fix_k_local}~\rm(i).
To prove Theorem~\ref{thm:alpha_beta_system_stability} \rm{(iii)}, we apply Proposition~\ref{hopf}.
The determinant of the Jacobian matrix at $(1,1)$ is given by $\al\be - \al + 1$, and its trace amounts to $\al+\be-2$.
For vanishing trace, we obtain
\begin{align*}
d_1 = - [(1-\al)(-\be)(1-\be)-(\al-1)(-1) \,0] = (\al-1)^2(\al-2) .
\end{align*}
Since pairs $(\al,\be)$ with $\al+\be-2=0$ and $\al\be - \al + 1>0$
lie on the line segment between $\left(\frac{1-\sqrt{5}}{2},\frac{3+\sqrt{5}}{2}\right)$ and $\left(\frac{1+\sqrt{5}}{2},\frac{3-\sqrt{5}}{2}\right)$,
we find $\alpha<2$.
Hence, if $(\al,\be)\neq(1,1)$, then $d_1<0$, the equilibrium $(1,1)$ is asymptotically stable,
and the corresponding Andronov-Hopf bifurcation is supercritical.
Now, Theorem~\ref{thm:alpha_beta_system_stability}~\rm{(ii)} follows from Proposition~\ref{fix_k_local} \rm{(ii)} and Theorem~\ref{thm:alpha_beta_system_stability} \rm{(iii)}.
Finally, we consider $(\al,\be)=(1,1)$, and hence $a_1=b_3=0$ and $a_3=b_1=-1$.
Theorem~\ref{thm:alpha_beta_system_stability}~\rm{(iv)} is a direct consequence of the discussion in Subsection \ref{subsubsec:first_integral}.

\subsubsection{Global behavior: Theorem~\ref{thm:alpha_beta_system_stability} \rm(v)}
\label{subsubsec:alpha_beta_system_glob_stab}

It remains to characterize the global asymptotic stability of the unique positive equilibrium $(1,1)$.
We assume $\al\be-\al+1>0$, a necessary condition for local asymptotic stability.

If $\al \leq 1$, $\be \leq 1$, and $(\al,\be)\neq (1,1)$, then Theorem~\ref{all_k_global} implies global asymptotic stability.
Indeed, for $\al,\be < 1$ we have $a_1 < 0 < b_3$,
for $\al<1=\be$ we have $a_1 < 0 = b_3$, $a_3<0$, $b_1 \leq -1$,
and for $\be < 1=\al$ we have $a_1 = 0 < b_3$, $a_3\leq -1$, $b_1<0$. (In the latter case, $\be>0$ since $\al\be-\al+1>0$ and $\al=1$.)

If $\al \geq 1$ and $\be \geq 1$, then we do not even have local asymptotic stability, see Theorem~\ref{thm:alpha_beta_system_stability}~\rm(ii).


If $\al < 1$ and $\be > 1$ then Lemma~\ref{lemma:a1_neg_b1_neg_a3_neg_b3_neg} precludes global asymptotic stability.
Indeed, we have $a_1=\al-1<0$, $b_1=-\beta<0$, $a_3=-1<0$, $b_3=1-\be<0$, and
\begin{align*}
\det C - (a_3+b_3) = -\al\be + \al - 1 + \be = (1-\al)(\be-1) > 0.
\end{align*}

If $\al>1$, $\be<1$, and $\be < \al-1$ then Lemma~\ref{lemma:a1_pos_b1_neg_a3_neg_b3_pos} precludes global asymptotic stability.
Indeed, we have $a_1=\al-1>0$, $b_1=-\beta<0$ (since $\al\be-\al+1>0$ and $\al>1$), $a_3=-1<0$, $b_3=1-\be>0$, and $a_1+b_1 = \al - 1 - \be > 0$.

It remains to consider the case $\al > 1$ and $\al-1 \leq \be < 1$.
Now, if $\be > 2-\al$, then the trace of the Jacobian matrix at $(1,1)$ is positive, and we do not even have local asymptotic stability.
Therefore, for the rest of this section, we assume
\begin{align} \label{eq:triangle}
1 < \al \leq \frac{3}{2} \text{ and } \al-1 \leq \be \leq 2 - \al
\end{align}
and show that the equilibrium $(1,1)$ is globally asymptotically stable in this region of the $(\al,\be)$-plane,
see the light green triangle in Figure~\ref{fig:stability_diagram_fixk}.

For the exponents in the $x$-nullcline $y=x^{\frac{\al-1}{\be}}$ and the $y$-nullcline $y=x^{\frac{1}{1-\be}}$,
the assumptions~\eqref{eq:triangle} imply $0 < \frac{\al-1}{\be} \leq 1 < \frac{1}{1-\be}$.
For a typical phase portrait, see Figure~\ref{fig:streamplot_triangle}.

Note that a solution can approach the boundary only at the origin.
Indeed, the $x$-axis for $x>0$ is repelling, 
while the $y$-axis for $y>0$ is an orbit (if we extend the state space of the ODE from the positive quadrant to the non-negative quadrant),
and solutions starting at $(0,y_0)$ with $y_0>0$ are unique backward in time.


To prove global asymptotic stability, we have to show that no solution starting in the positive quadrant
\begin{enumerate}[$\bullet$]
\item approaches the origin,
\item is unbounded or
\item periodic.
\end{enumerate}
This is guaranteed by the following three results.
In Lemma~\ref{lemma:triangle_dulac}, we construct a Dulac function
which allows to rule out periodic solutions by the Bendixson-Dulac test.
Lemma~\ref{lemma:triangle_no_escape} ensures that all orbits are bounded,
and Lemma~\ref{lemma:triangle_no_approach_origin} precludes convergence to the origin.

\begin{lem} \label{lemma:triangle_dulac}
Let $f(x,y) = x^\al - xy^\be$ and $g(x,y) = xy^\be -y$ denote the right-hand side of the ODE~\eqref{eq:ode_two_exponents},
and let $J$ be the Jacobian matrix at $(1,1)$.
Assume \eqref{eq:triangle} and let 
\begin{align*}
h(x,y) = x^{-p}y^{-q} \text{ with } p = \al - \frac{(\al-1)\be}{\det J} , \, q = \be + \frac{(\al-1)^2\be}{\det J}
\end{align*}
and $v = \frac{\dive(hf,hg)}{h}$.
Then, for $(x,y)\in\R^2_+$,
\begin{alignat*}{3}
& v(x,y) < v(1,1) < 0 \text{ for } (x,y)\neq (1,1), \quad && \text{if } 1<\al<\frac{3}{2} \text{ and } \al-1 < \be < 2-\al, \\
& v(x,y) < v(1,1) = 0 \text{ for } (x,y)\neq (1,1), && \text{if } 1<\al<\frac{3}{2} \text{ and } \be = 2-\al, \\
& v(x,y) < v(1,1) < 0 \text{ for } x \neq y, && \text{if } 1<\al<\frac{3}{2} \text{ and } \al-1 = \be, \\
& v(x,y) < v(1,1) = 0 \text{ for } x \neq y, && \text{if } \al=\frac{3}{2} \text{ and } \be = \frac{1}{2}.
\end{alignat*}
\end{lem}
\begin{proof}
We first explain the choice of $p$ and $q$ in the definition of $h$. Since
\begin{align*}
v = \frac{\dive(hf,hg)}{h} 
= \partial_1 f + \partial_2 g + f\frac{\partial_1 h}{h} + g\frac{\partial_2 h}{h},
\end{align*}
we have $v(1,1) = \tr J$. 
Hence, for $\tr J = \al +\be- 2 =0$ (at the supercritical Andronov-Hopf bifurcation),
we have $v \leq 0$ on $\R^2_+$ only if $v$ has a local maximum at $(1,1)$.
A short calculation shows that $(\grad v)(1,1)=(0,0)$ uniquely determines $p$ and $q$.
For $\al + \be - 2<0$, other $p$ and $q$ may also be appropriate.

We now turn to the proof of the lemma.
We assume~\eqref{eq:triangle} and find
\begin{align*}
v(x,y) &= (\al - p) x^{\al - 1} + (p - 1)y^\be + (\be - q)xy^{\be-1} + q - 1 \quad \text{and} \\
(\partial_1v)(x,y) &= (\al-p)(\al-1)x^{\al-2} + (\be-q)y^{\be-1} \\
&= \frac{(\al-1)^2\be}{\det J}(x^{\al-2} - y^{\be-1}) ,
\end{align*}
where $\frac{(\al-1)^2\be}{\det J} > 0$.

For fixed $y > 0$, the function $x \mapsto v(x,y)$ has its maximum at $x = y^{\frac{1-\be}{2-\al}}$ and
\begin{align*}
\max_{(x,y)\in \R^2_+} v(x,y) = \max_{y > 0} \tilde v(y) \quad \text{with } \tilde v(y) = v \Big(y^\frac{1-\be}{2-\al},y\Big).
\end{align*}
We find
\begin{align*}
\tilde v(y) &= (\al-p+\be-q) y^{\frac{(\al-1)(1-\be)}{2-\al}}+(p-1) y^{\be} + q-1\\
&= \frac{(\al-1)(2-\al)\be}{\det J} y^{\frac{(\al-1)(1-\be)}{2-\al}} - \frac{(\al-1)^2(1-\be)}{\det J} y^{\be} + q-1 \quad \text{and} \\
\tilde v'(y) &= \frac{(\al-1)^2\be(1-\be)}{\det J}\left(y^{\frac{(\al-1)(1-\be)}{2-\al}-1} - y^{\be-1} \right) ,
\end{align*}
where $\frac{(\al-1)^2\be(1-\be)}{\det J} > 0$. 

If $\be > \al-1$, then $\frac{(\al-1)(1-\be)}{2-\al} - 1 < \be-1$
and hence $\tilde v(y)<v(1,1)$ for all $y>0$.
If $\be=\al-1$, then $\frac{(\al-1)(1-\be)}{2-\al} - 1 = \be - 1$
and hence $v(y,y)=v(1,1)$ for all $y>0$ and $v(x,y) < v(1,1)$ for $x\neq y$.
\qed
\end{proof}

\begin{lem} \label{lemma:triangle_no_escape}
Assume \eqref{eq:triangle}.
Then all solutions of the ODE~\eqref{eq:ode_two_exponents} starting in the positive quadrant are bounded.
\end{lem}

\begin{proof}
We claim that solutions of the ODE~\eqref{eq:ode_two_exponents} starting below the $x$-nullcline with $x>1$ eventually cross it,
see the left panel in Figure~\ref{fig:triangle_no_escape_no_approach_origin}.
Solutions above the $x$-nullcline with $x>1$ eventually cross the $y$-nullcline after which $y$ decreases,
see Figure~\ref{fig:streamplot_triangle}.
Since orbits spiraling outwards are precluded by Lemma~\ref{lemma:triangle_dulac}, all orbits are bounded.

To prove our claim,
we consider a starting point below the $x$-nullcline
and show that there exists a curve below the point
on which the vector field points upward and which eventually intersects the $x$-nullcline.
Formally,
we show that there exists $c_0$ with $0<c_0<1$ such that for all $c$ with $0<c \le c_0$ the following statement holds.
Along the curve $y=cx^{\frac{1}{1-\be}}$,
\begin{align*} 
\dd{y}{x} = \frac{\dot y}{\dot x} = \frac{g(x,cx^{\frac{1}{1-\be}})}{f(x,cx^{\frac{1}{1-\be}})} > \frac{c}{1-\be}x^{\frac{1}{1-\be}-1} = \dd{}{x} cx^{\frac{1}{1-\be}}
\quad \text{for all } 0 < x < c^{-\frac{\be(1-\be)}{\det J}} ,
\end{align*}
where $f(x,y) = x^\al - xy^\be$ and $g(x,y) = xy^\be -y$ denote the right-hand side of the ODE~\eqref{eq:ode_two_exponents}.

The assumptions~\eqref{eq:triangle} imply $0 < \frac{\al-1}{\be} \leq 1 < \frac{1}{1-\be}$.
Hence, the $x$-nullcline $y = x^{\frac{\al-1}{\be}}$ and the scaled $y$-nullcline $y = cx^{\frac{1}{1-\be}}$ intersect at $x=0$ and $x = c^{-\frac{\be(1-\be)}{\det J}}$.

A short calculation shows that the statement in question is equivalent to
\begin{align*}
p(x) > 0 \quad \text{ for all } 0 < x < c^{-\frac{\be(1-\be)}{\det J}}
\end{align*}
with
\begin{align*}
p(x) = x^{\frac{\be}{1-\be}} - x^{\al-1} c^{-\be} + (1-\be)(c^{-1} - c^{-\be}) .
\end{align*}
The function $p$ has its minimum at
\begin{align*}
\bar x = \left(\frac{\be-\det J}{\be}\right)^{\frac{1-\be}{\det J}} c^{-\frac{\be(1-\be)}{\det J}} ,
\end{align*}
where $0 < \bar x < c^{-\frac{\be(1-\be)}{\det J}}$ since $\det J<\be$.
We find
\begin{align*} 
p(\bar x)=-\frac{\det J}{\be-\det J} \left(\frac{\be-\det J}{\be}\right)^\frac{\be}{\det J} c^{-\frac{\be^2}{\det J}} + (1-\be)(c^{-1} - c^{-\be}),
\end{align*}
where the first summand is negative and the second summand is positive.

If $\be  > \al-1$, then $-\frac{\be^2}{\det J}>-1$. Therefore, $p(\bar x) \to \infty$ as $c \to 0$.
If $\be  = \al-1$, then $\det J=\be^2$. Therefore, 
\begin{align*}
p(\bar x) = \left(-\be(1-\be)^{\frac{1}{\be}-1}+(1-\be)\right)c^{-1}-(1-\be)c^{-\be} ,
\end{align*}
where the first summand is positive since $0<\be\le \frac{1}{2}$, and $p(\bar x) \to \infty$ as $c \to 0$.
\qed
\end{proof}

\begin{lem} \label{lemma:triangle_no_approach_origin}
Assume \eqref{eq:triangle}.
Then no solution of the ODE~\eqref{eq:ode_two_exponents} starting in the positive quadrant approaches the origin.
\end{lem}
\begin{proof}
We show that solutions of the ODE~\eqref{eq:ode_two_exponents} starting above the $x$-nullcline with $y>0$ small enough eventually cross it,
see the right panel in Figure~\ref{fig:triangle_no_escape_no_approach_origin}.

Fix $\gamma > \frac{\al-1}{\be}$.
Every starting point $(x_0,y_0)$ above the $x$-nullcline $y=x^\frac{\al-1}{\be}$ with $y_0>0$ small enough
lies on a curve $y = c x^\gamma$ with $c>1$
on which the vector field points downward
and which eventually intersects the $x$-nullcline.
Indeed, the intersection occurs at $x=c^{-1/(\gamma - \frac{\al-1}{\be})}$ with $c=y_0/x_0^\gamma$.
Finally, we prove that, along any curve $y = c x^\gamma$ with $c>1$,
\begin{align*}
\dd{y}{x} = \frac{\dot y}{\dot x} = \frac{xy^\be-y}{x^\al-xy^\be} > \gamma\frac{y}{x} = \dd{}{x} c x^\gamma \quad \text{for all } x^\frac{\al-1}{\be} < y < \min(1,\gamma^{-\frac{1}{\be}})
\end{align*}
or, equivalently,
\begin{align*}
h(x,y) < 1 - \gamma y^\be \quad \text{for all } x^\frac{\al-1}{\be} < y < \min(1,\gamma^{-\frac{1}{\be}})
\end{align*}
with
\begin{align*}
h(x,y) = xy^{\be-1}-\gamma x^{\al-1} .
\end{align*}
We have
\begin{alignat*}{3}
& h(0,y)=0 < 1-\gamma y^\be, && \text{if } 0 < y < \gamma^{-\frac{1}{\be}} \quad \text{and} \\
& h(y^\frac{\be}{\al-1},y) = y^{\frac{\be}{\al-1}+\be-1} -\gamma y^\be < 1 - \gamma y^\be, \quad && \text{if } 0 < y < 1 .
\end{alignat*}
Since the function $x \mapsto h(x,y)$ is convex for $1 < \al < 2$,
we obtain $h(x,y) < 1-\gamma y^\be$ if $0 < x < y^{\frac{\be}{\al-1}}$ and $0 < y < \min(1,\gamma^{-\frac{1}{\be}})$. \qed
\end{proof}


\begin{acknowledgements}
We thank Georg Regensburger, Valerij Romanovskij, and J\'anos T\'oth for fruitful discussions.
BB and SM were supported by the Austrian Science Fund (FWF), project P28406.
\end{acknowledgements}


\bibliographystyle{abbrv}
\bibliography{stability}

\begin{thebibliography}{10}

\bibitem{andronov:leontovich:1937}
A.~A. Andronov and E.~A. Leontovich.
\newblock Some cases of dependence of limit cycles on a parameter.
\newblock {\em Uchenye zapiski Gorkovskogo Universiteta}, 6:3--24, 1937.

\bibitem{andronov:leontovich:gordon:maier:1973}
A.~A. Andronov, E.~A. Leontovich, I.~I. Gordon, and A.~G. Ma{\u\i}er.
\newblock {\em Theory of bifurcations of dynamic systems on a plane}.
\newblock Halsted Press [A division of John Wiley \& Sons], New York-Toronto,
  Ont.; Israel Program for Scientific Translations, Jerusalem-London, 1973.
\newblock Translated from the Russian.

\bibitem{dancso:farkas:farkas:szabo:1991}
A.~Dancs{\'o}, H.~Farkas, M.~Farkas, and G.~Szab{\'o}.
\newblock Investigations into a class of generalized two-dimensional
  {L}otka-{V}olterra schemes.
\newblock {\em Acta Appl. Math.}, 23(2):103--127, 1991.

\bibitem{farkas:noszticzius:1985}
H.~Farkas and Z.~Noszticzius.
\newblock Generalized {L}otka-{V}olterra schemes and the construction of
  two-dimensional explodator cores and their {L}iapunov functions via
  ``critical'' {H}opf bifurcations.
\newblock {\em J. Chem. Soc. Faraday Trans. II}, 81(10):1487--1505, 1985.

\bibitem{farkas:1994}
M.~Farkas.
\newblock {\em Periodic Motions}, volume 104 of {\em Applied Mathematical
  Sciences}.
\newblock Springer-Verlag, New York, 1994.

\bibitem{frommer:1934}
M.~Frommer.
\newblock \"{U}ber das {A}uftreten von {W}irbeln und {S}trudeln (geschlossener
  und spiraliger {I}ntegralkurven) in der {U}mgebung rationaler
  {U}nbestimmtheitsstellen.
\newblock {\em Math. Ann.}, 109(1):395--424, 1934.

\bibitem{guckenheimer:holmes:1990}
J.~Guckenheimer and P.~Holmes.
\newblock {\em Nonlinear oscillations, dynamical systems, and bifurcations of
  vector fields}, volume~42 of {\em Applied Mathematical Sciences}.
\newblock Springer-Verlag, New York, 1990.
\newblock Revised and corrected reprint of the 1983 original.

\bibitem{hopf:1942}
E.~Hopf.
\newblock Abzweigung einer periodischen {L}{\"o}sung von einer station{\"a}ren
  {L}{\"o}sung eines {D}ifferentialsystems.
\newblock {\em Ber. S{\"a}chs. Akad. Wiss. Leipzig, Math.-Phys. Kl.}, 94:1--22,
  1942.

\bibitem{llibre:2016}
J.~Llibre.
\newblock A counterexample to a result on {L}otka-{V}olterra systems.
\newblock {\em Acta Applicandae Mathematicae}, 142(1):123--125, 2016.

\bibitem{lotka:1910}
A.~J. Lotka.
\newblock Contribution to the theory of periodic reactions.
\newblock {\em J. Phys. Chem.}, 14(3):271--274, 1910.

\bibitem{lotka:1920:a}
A.~J. Lotka.
\newblock Analytical note on certain rhythmic relations in organic systems.
\newblock {\em Proc. Natl. Acad. Sci.}, 6(7):410--415, 1920.

\bibitem{lotka:1920:b}
A.~J. Lotka.
\newblock Undamped oscillations derived from the law of mass action.
\newblock {\em J. Am. Chem. Soc.}, 42:1595--1599, 1920.

\bibitem{mueller:regensburger:2012}
S.~M\"uller and G.~Regensburger.
\newblock Generalized mass action systems: {C}omplex balancing equilibria and
  sign vectors of the stoichiometric and kinetic-order subspaces.
\newblock {\em SIAM J. Appl. Math.}, 72:1926--1947, 2012.

\bibitem{mueller:regensburger:2014}
S.~M\"uller and G.~Regensburger.
\newblock Generalized mass-action systems and positive solutions of polynomial
  equations with real and symbolic exponents.
\newblock In V.~P. Gerdt, W.~Koepf, E.~W. Mayr, and E.~H. Vorozhtsov, editors,
  {\em Computer Algebra in Scientific Computing. Proceedings of the 16th
  International Workshop (CASC 2014)}, volume 8660 of {\em Lecture Notes in
  Comput. Sci.}, pages 302--323, Berlin/Heidelbergx, 2014. Springer.

\bibitem{yin:voit:2008}
W.~Yin and E.~O. Voit.
\newblock Construction and customization of stable oscillation models in
  biology.
\newblock {\em Journal of Biological Systems}, 16(04):463--478, 2008.

\end{thebibliography}

\section*{Appendix A: First focal value}

At an Andronov-Hopf bifurcation of a planar ODE,
the Jacobian matrix has zero trace and positive determinant.

For a planar ODE
\begin{align} \label{eqn_ode}
\dot x &= f(x,y) , \\ 
\dot y &= g(x,y) \nonumber 
\end{align}
with equilbrium at the origin and Jacobian matrix in normal form,
\[
J = \begin{pmatrix} 0 & -\omega \\ \omega & 0 \end{pmatrix} ,
\]
the first focal value (erste Strudelgr\"o{\ss}e, cf.~\cite{frommer:1934}), called $D_1$ here,
is given by
\begin{align} \label{GH}
16 D_1 &= f_{30}+f_{12}+g_{21}+g_{03} \\
& \quad + \frac{1}{\omega} \left[ f_{11}(f_{20}+f_{02})-g_{11}(g_{20}+g_{02})-f_{20}g_{20}+f_{02}g_{02} \right] , \nonumber
\end{align}
where $f_{30}=(\partial^3f/\partial x^3)$, $f_{12}=(\partial^3f/\partial x\partial y^2)$, etc.
See, for example,
\cite[p.~252, eqn.~(71)]{andronov:leontovich:gordon:maier:1973},
\cite[p.~431, Lemma 7.2.7]{farkas:1994}, or
\cite[p.~152, eqn.~(3.4.11)]{guckenheimer:holmes:1990}.

For the planar ODE~\eqref{eqn_ode}
with equilibrium at the origin and Jacobian matrix
\[
J = \begin{pmatrix} a & b \\ c & -a \end{pmatrix} ,
\]
we define the linear transformation
\[
\begin{pmatrix}
\tilde x \\ \tilde y
\end{pmatrix}
=
T
\begin{pmatrix}
x \\ y
\end{pmatrix}
\quad \text{with} \quad
T=
\begin{pmatrix}
1 & 0 \\
-\frac{a}{\omega} & -\frac{b}{\omega} 
\end{pmatrix}
\]
and $\omega = \sqrt{\det J} = \sqrt{-a^2-bc}$.
Introducing $h \colon (x,y)^\mathsf{T} \mapsto (f(x,y),g(x,y))^\mathsf{T}$,
we obtain the transformed ODE system
\[
\begin{pmatrix}
\tilde x \\ \tilde y
\end{pmatrix}
^{\textstyle \cdot}
=
T
\begin{pmatrix}
x \\ y
\end{pmatrix}
^{\textstyle \cdot}
=
T \, h \!
\begin{pmatrix}
x \\ y
\end{pmatrix}
=
T \, h \! \left( T^{-1}
\begin{pmatrix}
\tilde x \\ \tilde y
\end{pmatrix}
\right)
\]
with Jacobian matrix $\tilde J = T J \, T^{-1}$ in normal form.
Using~\eqref{GH}, we compute the first focal value $D_1$ of the transformed system
expressed in terms of derivatives of the original system.
We obtain
\begin{align} \label{strudel}
16 \, b^2 \omega^2 D_1 
=& \, b \left\{ \omega^2 \left[
-2 a (f_{21} + g_{12}) + b (f_{30} + g_{21}) - c (f_{12} + g_{03})
\right] \right. \\
& \quad + ab \left[
+ f_{20}^2 - f_{20} g_{11} - f_{11} g_{20} - g_{20} g_{02} - 2 g_{11}^2
\right] \nonumber \\
& \quad + ac \left[
-f_{20} f_{02} - 2 f_{11}^2 - f_{11} g_{02} - f_{02} g_{11} + g_{02}^2
\right] \nonumber \\
& \quad
+(b c - 2 a^2) [f_{20} f_{11} - g_{11} g_{02}]
\nonumber \\
& \quad 
+ \left. b^2 g_{20} (f_{20} + g_{11}) - c^2 f_{02} (f_{11} + g_{02}) 
\right\} \nonumber ,
\end{align}
See again~\cite[p.~253, eqn.~(76)]{andronov:leontovich:gordon:maier:1973},
where series coefficients are used instead of partial derivatives.

If the equilibrium under consideration differs from the origin,
then the derivatives need to be evaluated at the equilibrium.


\section*{Appendix B: Figures}

In order to illustrate our analysis of the ODE~\eqref{ode_math} in Section~\ref{sec:proofs_of_main_results},
we present phase portraits and figures of forward invariant sets.
Thereby, the red curve is the $x$-nullcline, $x^{a_1}y^{b_1}=\frac{k_2}{k_1}$, while the green curve is the $y$-nullcline, $x^{a_3}y^{b_3}=\frac{k_3}{k_4}$.

\begin{figure}[h!t!]
\begin{center}
\includegraphics[scale=.35]{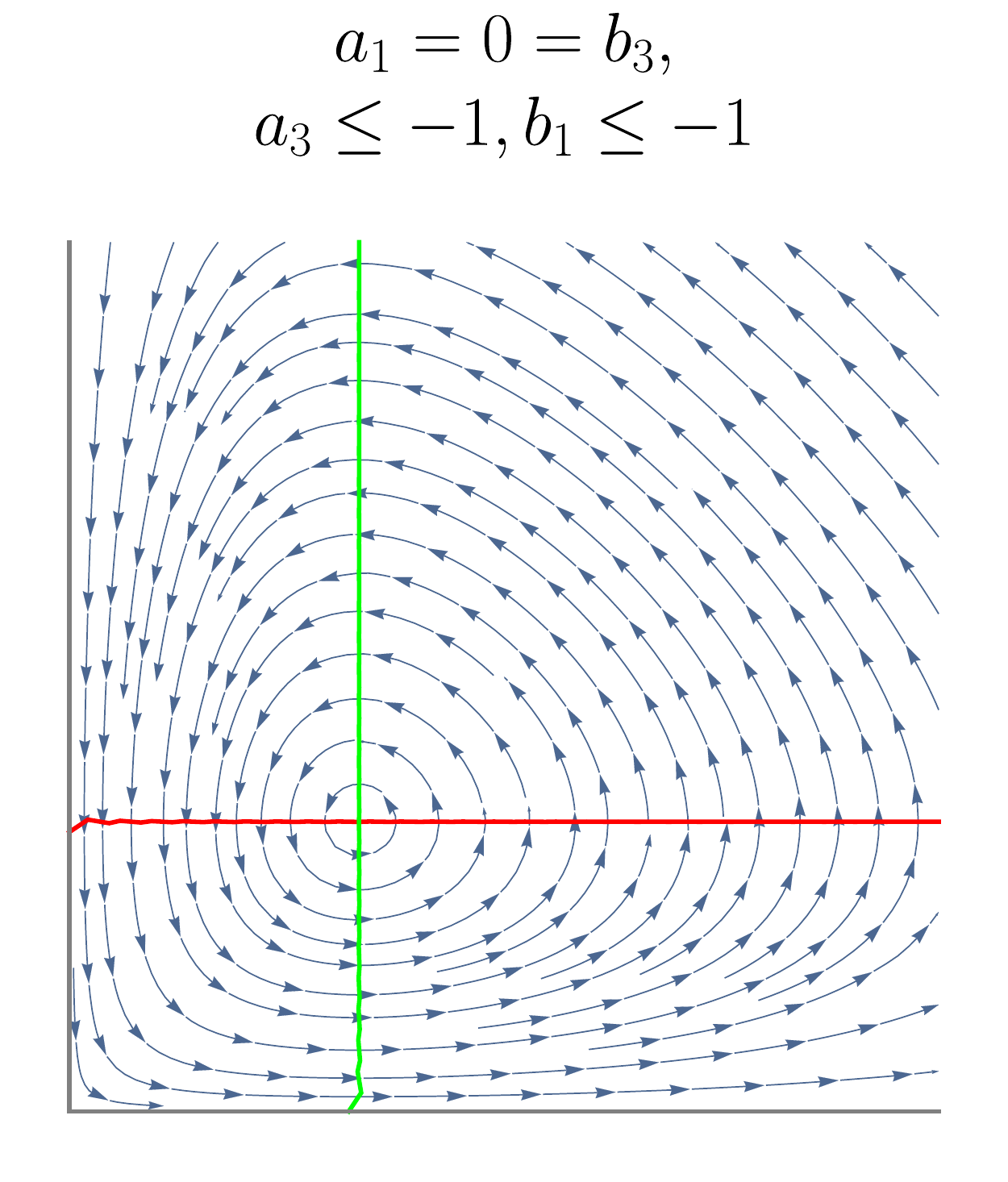}
\end{center}
\caption{Case $(a_1,b_3)=(0,0)$, $a_3\leq -1$, and $b_1\leq -1$. All solutions starting in $\R^2_+$ are periodic.}
\label{fig:streamplot_center_stay_away_from_the_boundary}
\end{figure}

\begin{figure}
\begin{center}
\begin{tabular}{cc}
\includegraphics[scale=.35]{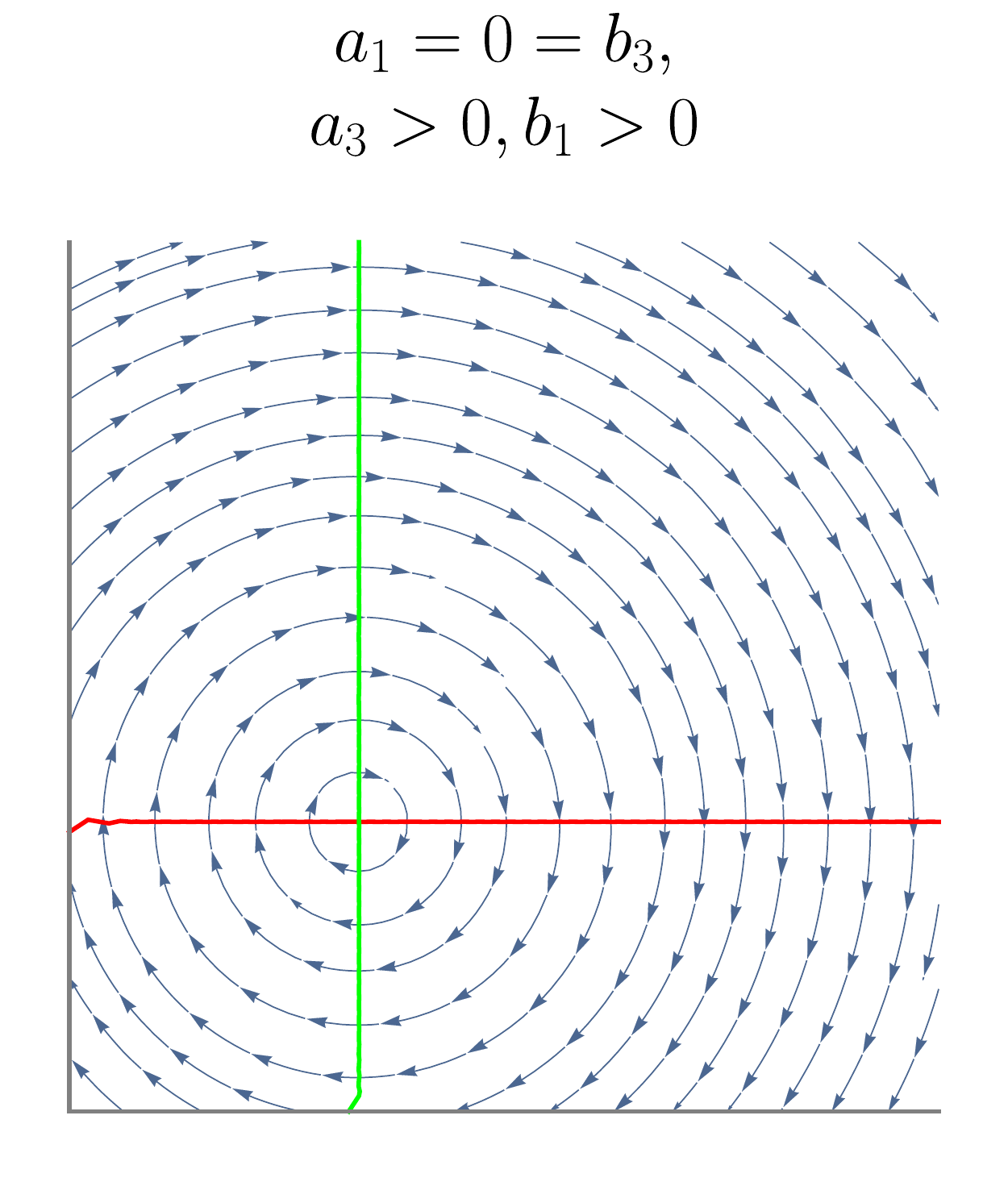} &
\includegraphics[scale=.35]{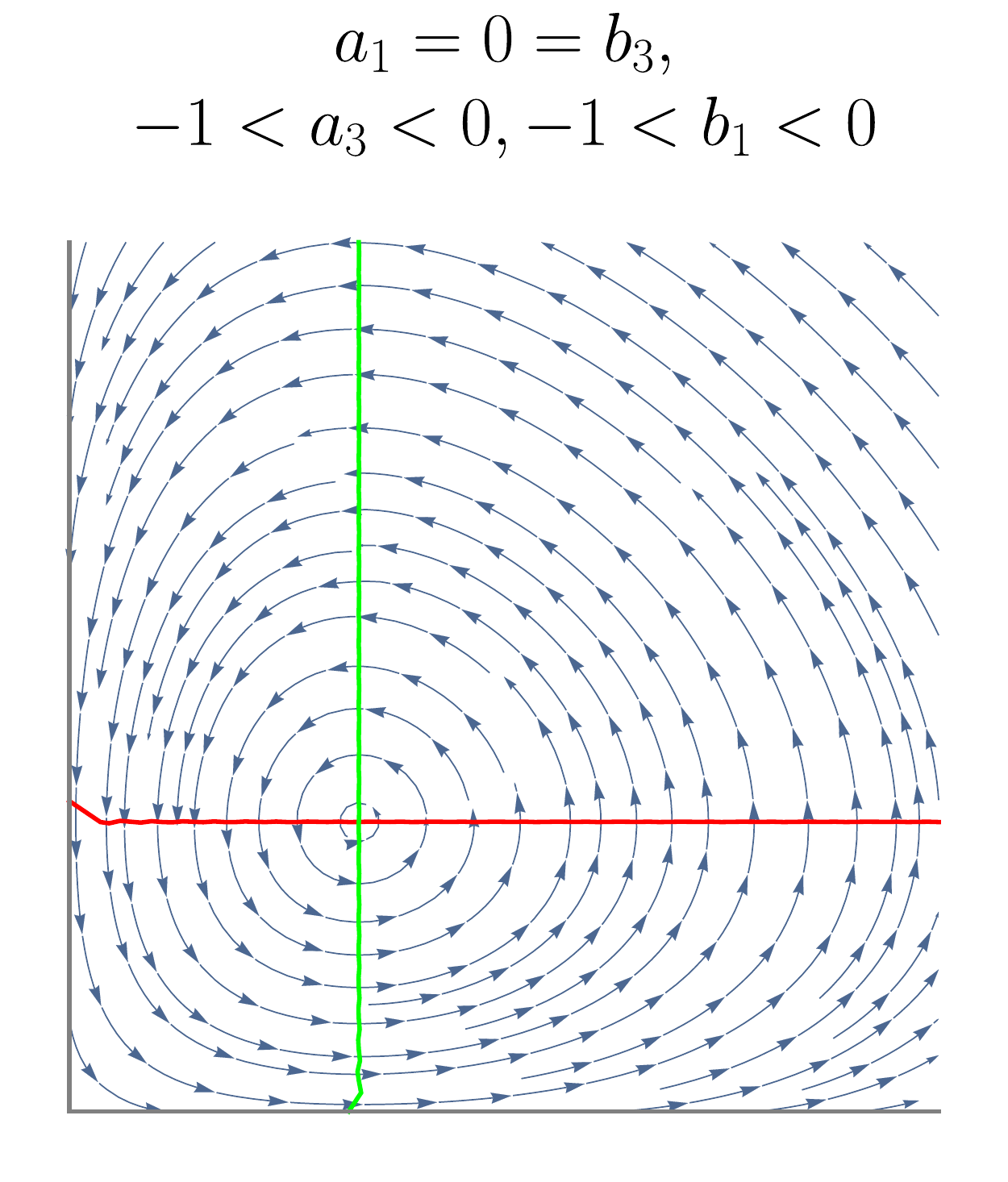} \\
\includegraphics[scale=.35]{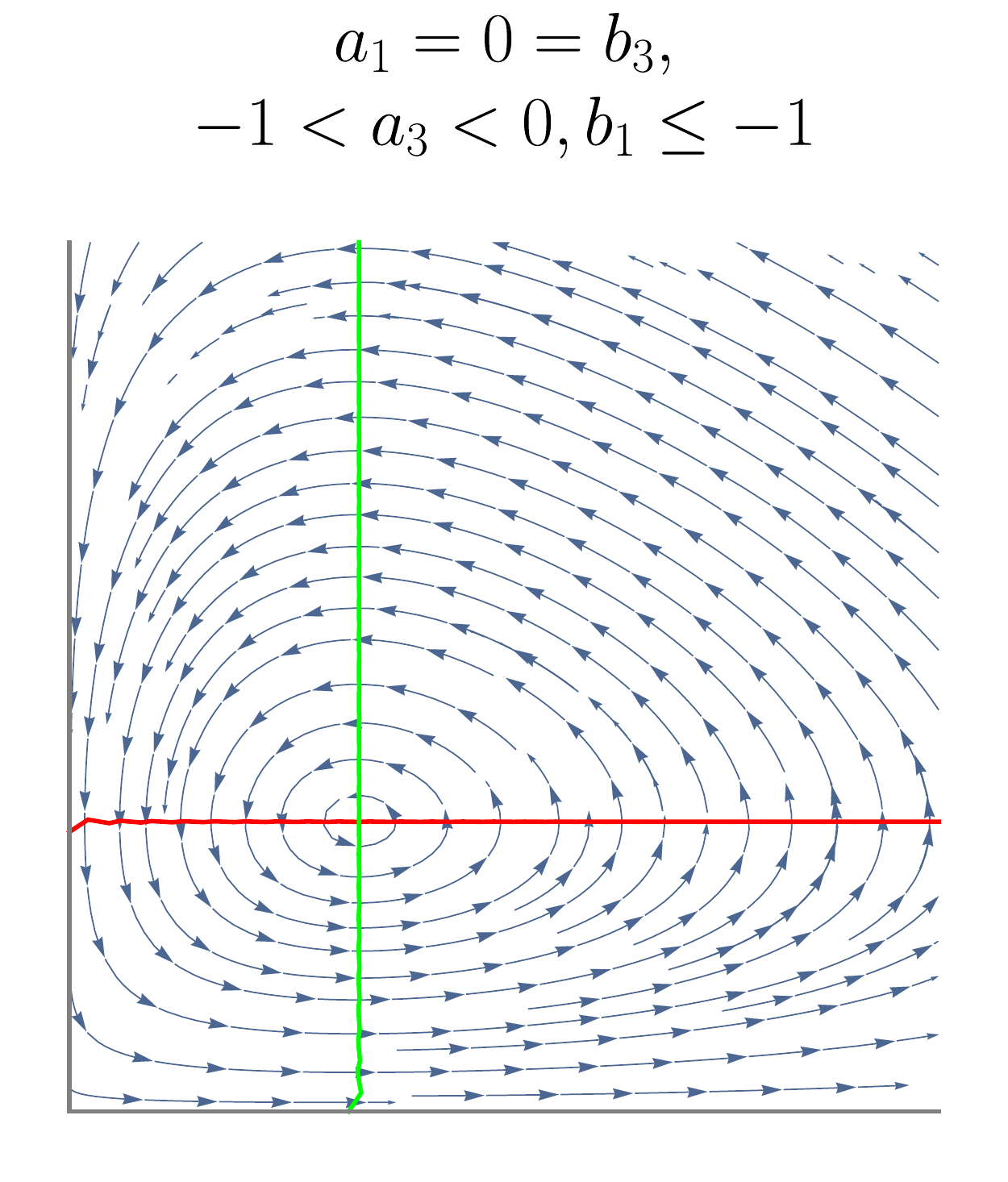} &
\includegraphics[scale=.35]{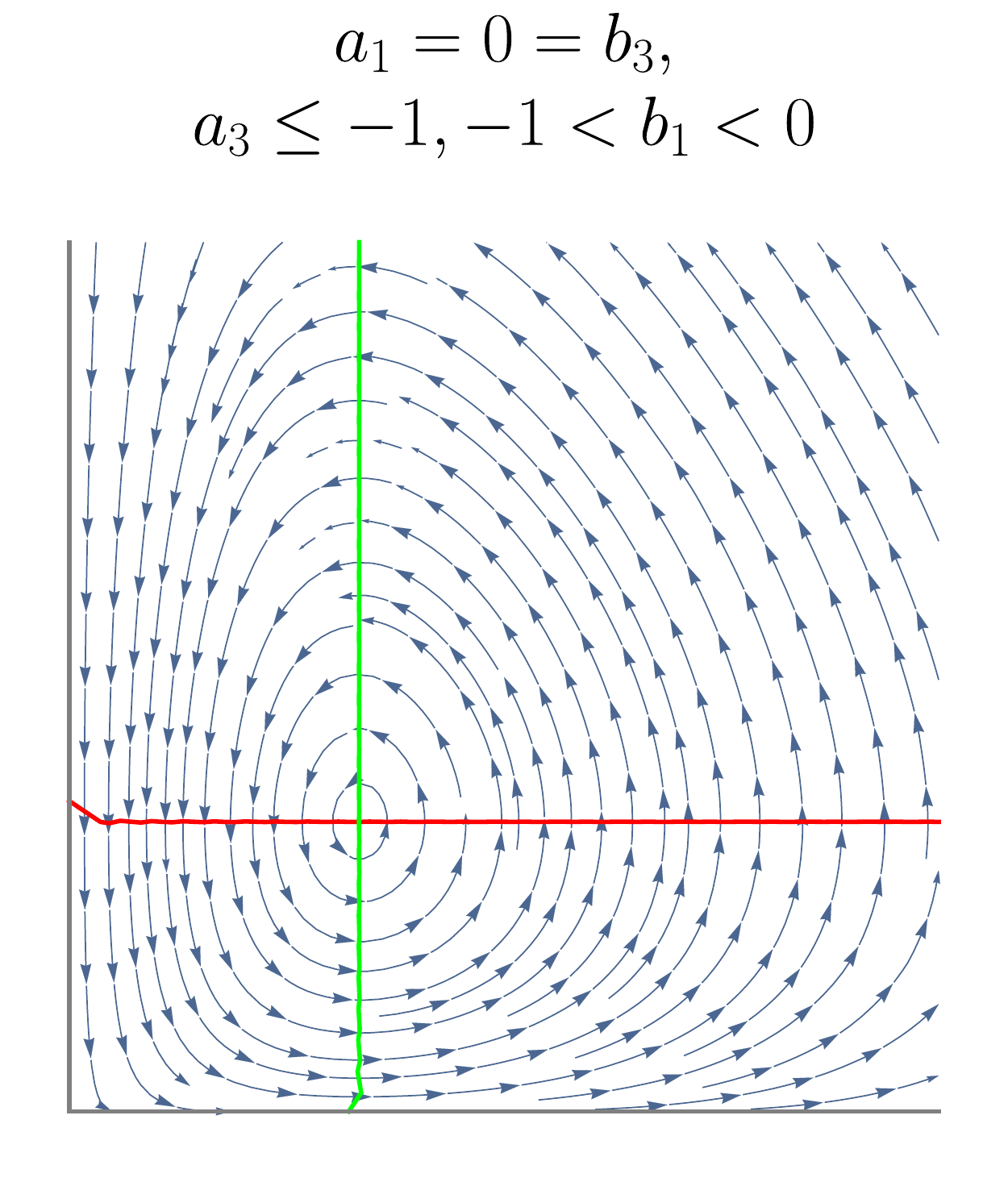}
\end{tabular}
\end{center}
\caption{Case $(a_1,b_3)=(0,0)$, $a_3b_1 > 0$, and at least one of $a_3\leq -1$ and $b_1\leq -1$ is violated.
Top left: $a_3>0$ and $b_1>0$ (both the $x$-axis and the $y$-axis are approached by some orbits).
Top right: $-1<a_3<0$ and $-1<b_1<0$ (both the $x$-axis and the $y$-axis are approached by some orbits).
Bottom left: $-1<a_3<0$ and $b_1\leq -1$ (the $y$-axis is approached by some orbits).
Bottom right: $a_3 \leq -1$ and $-1<b_1<0$ (the $x$-axis is approached by some orbits).}
\label{fig:streamplots_center_meet_the_boundary}
\end{figure}

\begin{figure}
\begin{center}
\begin{tabular}{ccc}
\includegraphics[scale=.28]{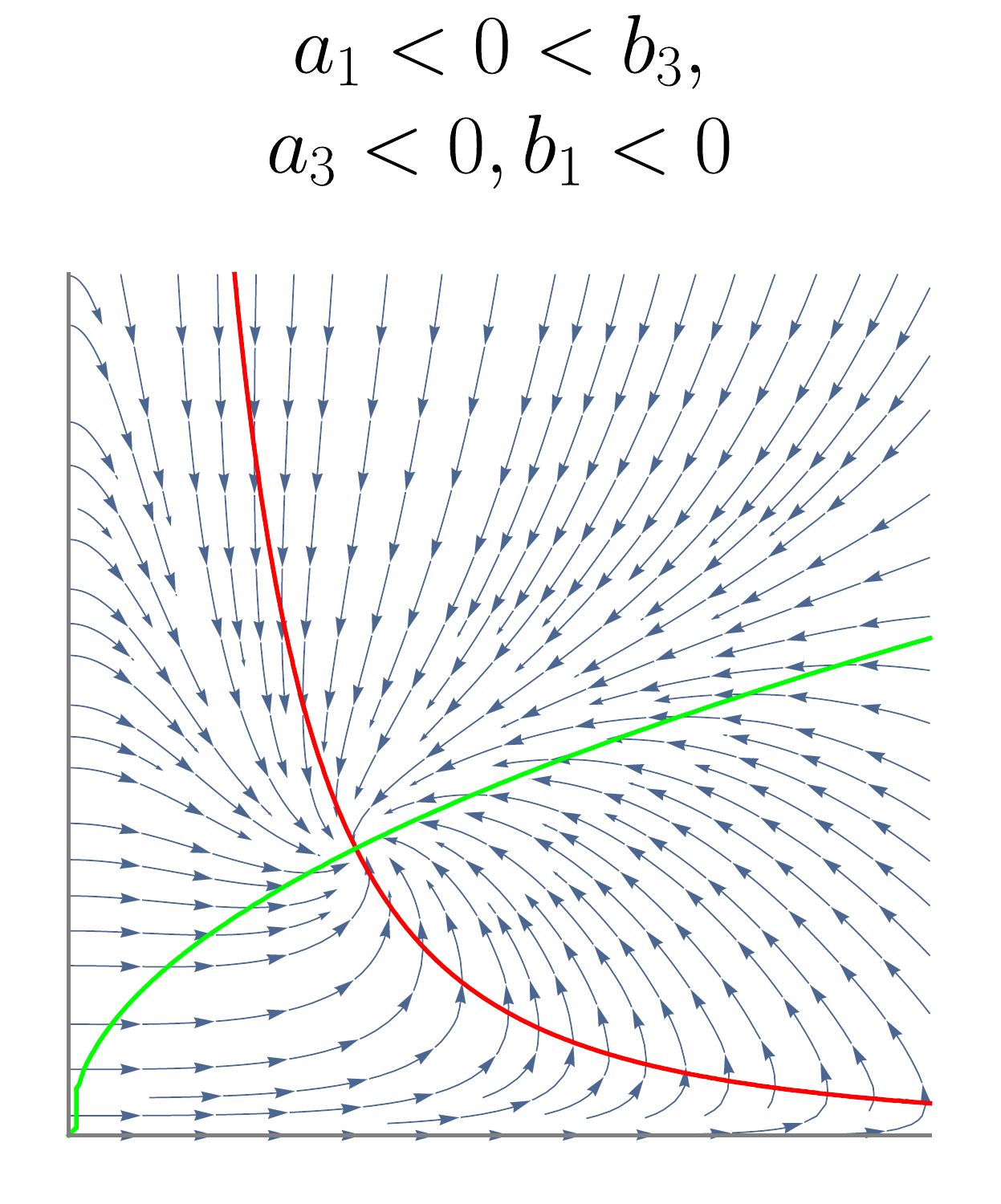} &
\includegraphics[scale=.28]{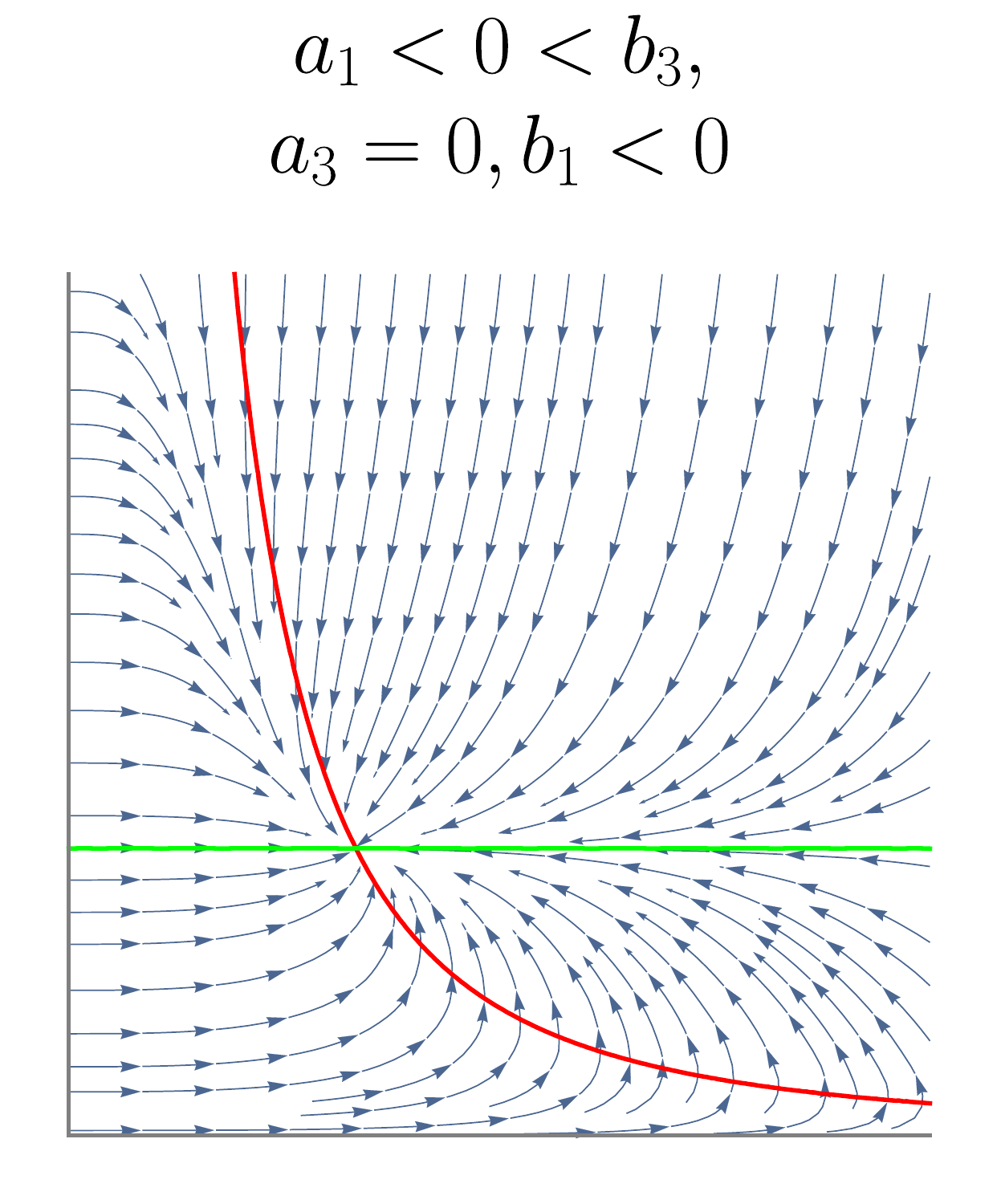} &
\includegraphics[scale=.28]{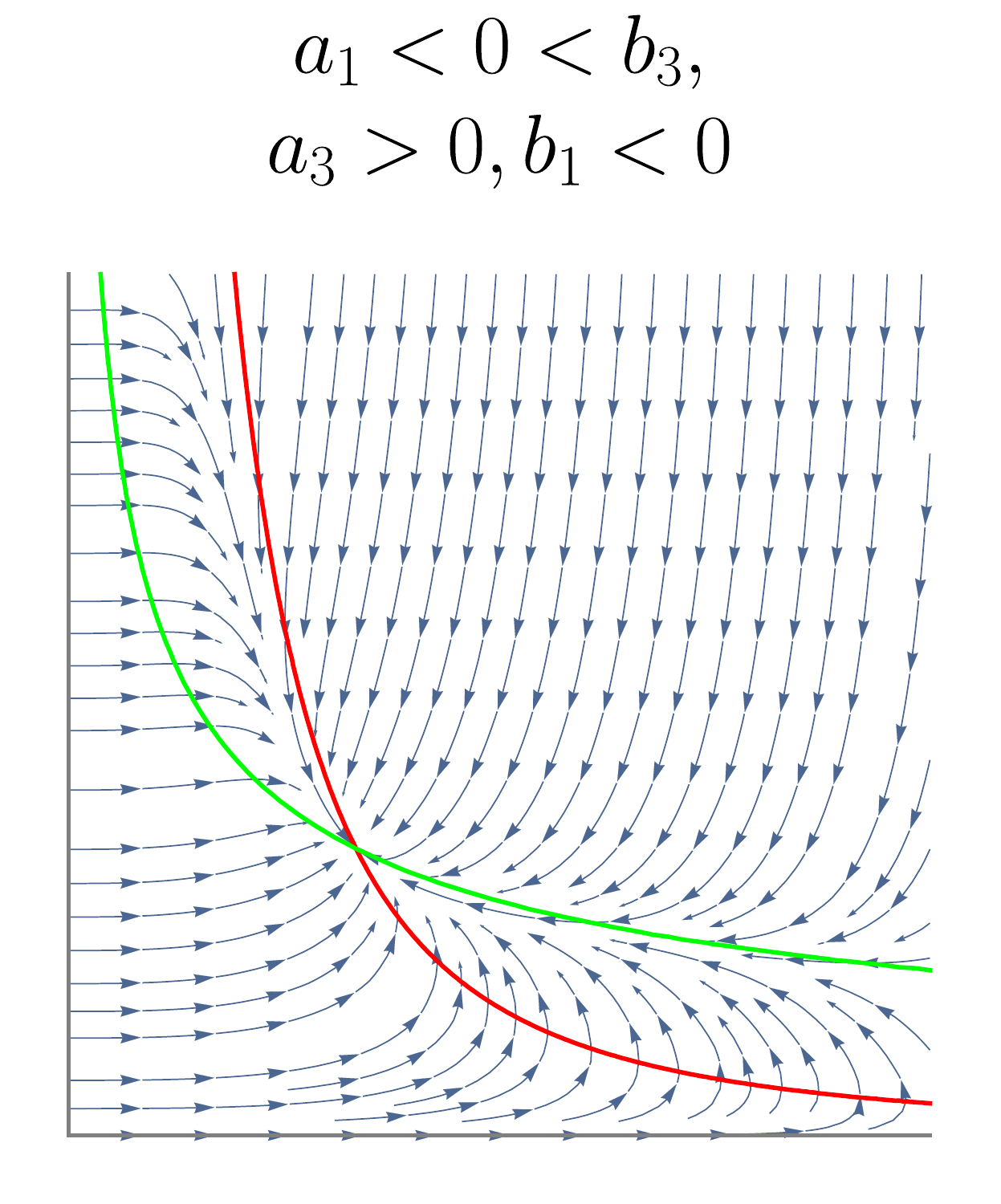} \\
\includegraphics[scale=.28]{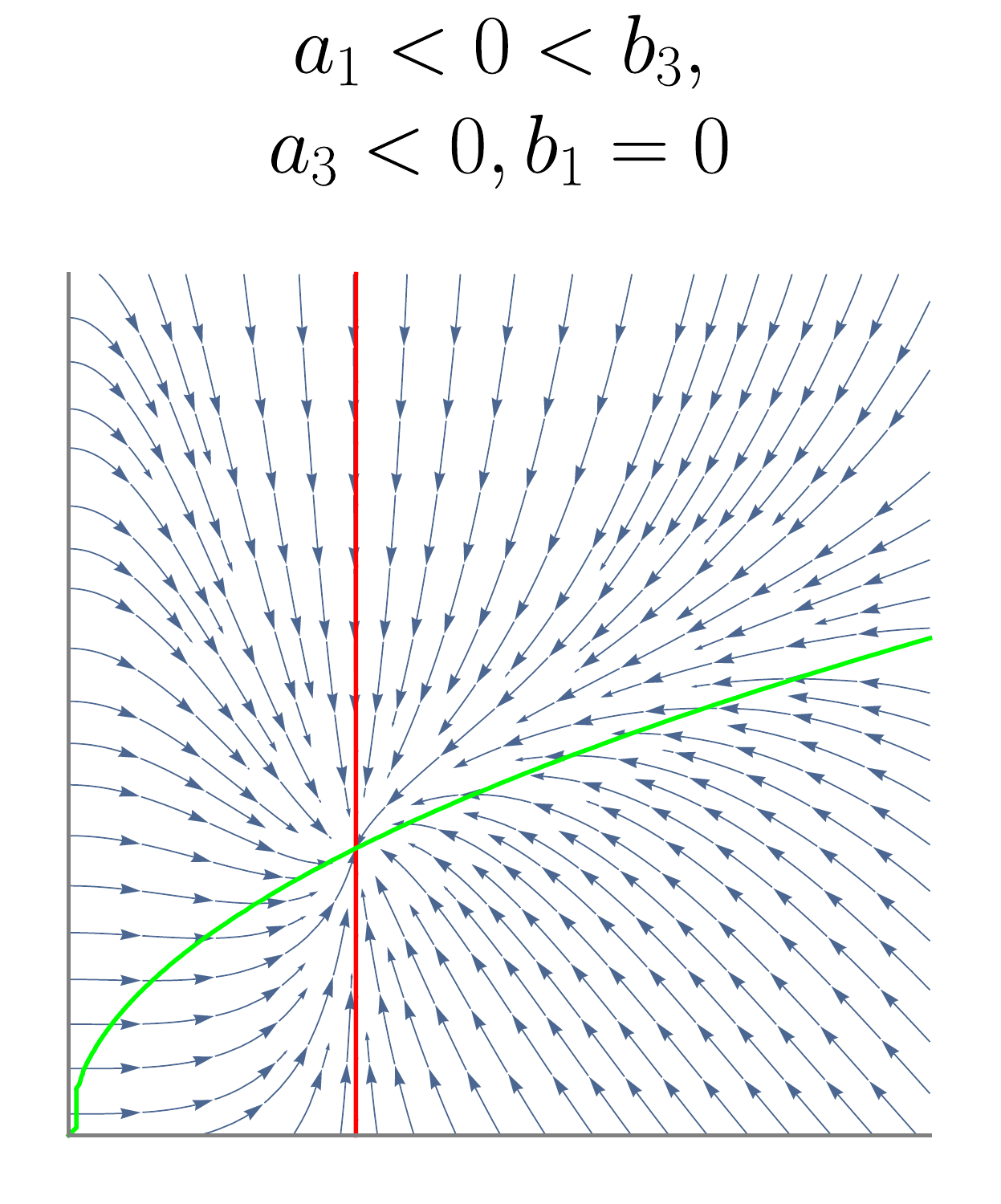} &
\includegraphics[scale=.28]{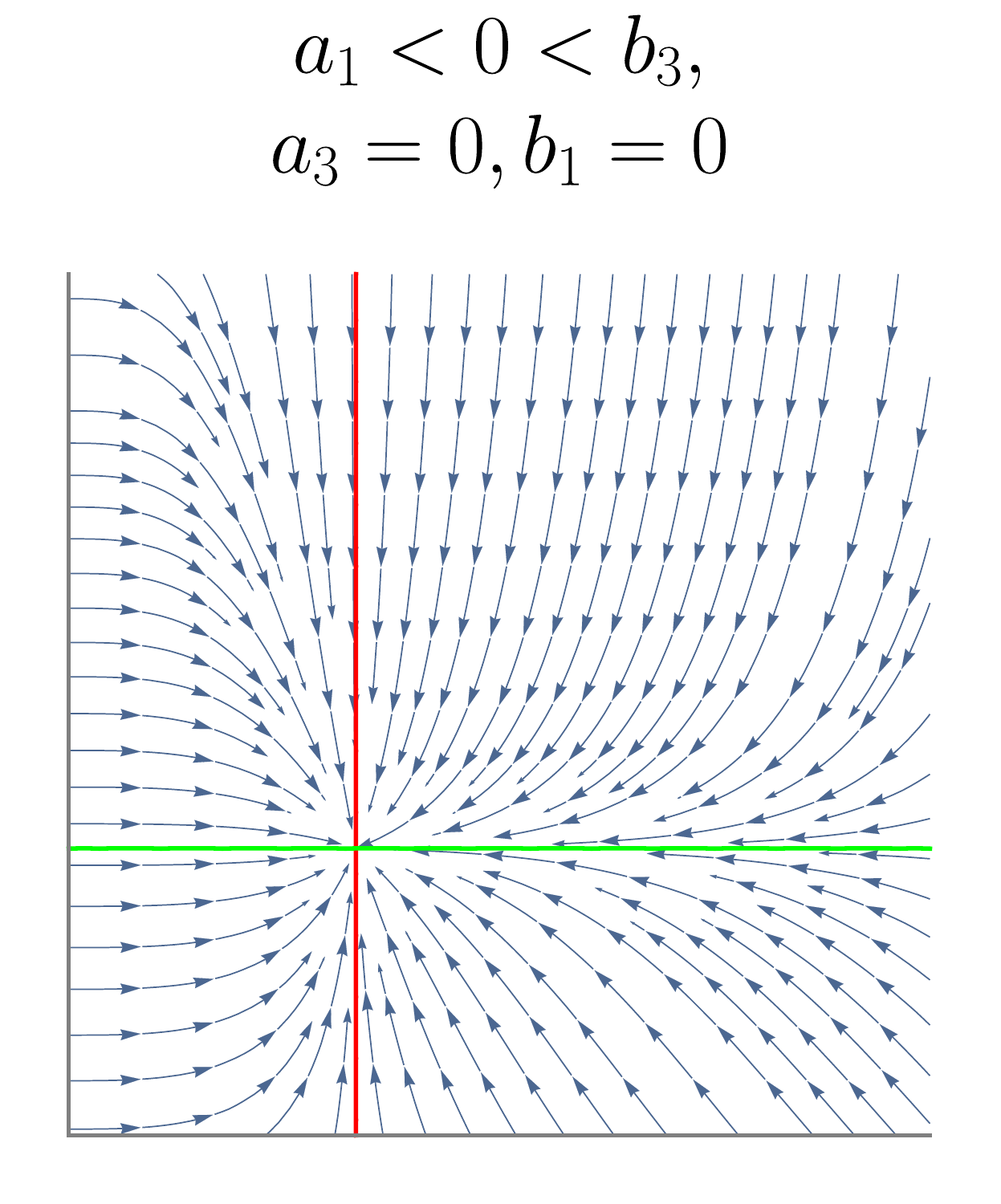} &
\includegraphics[scale=.28]{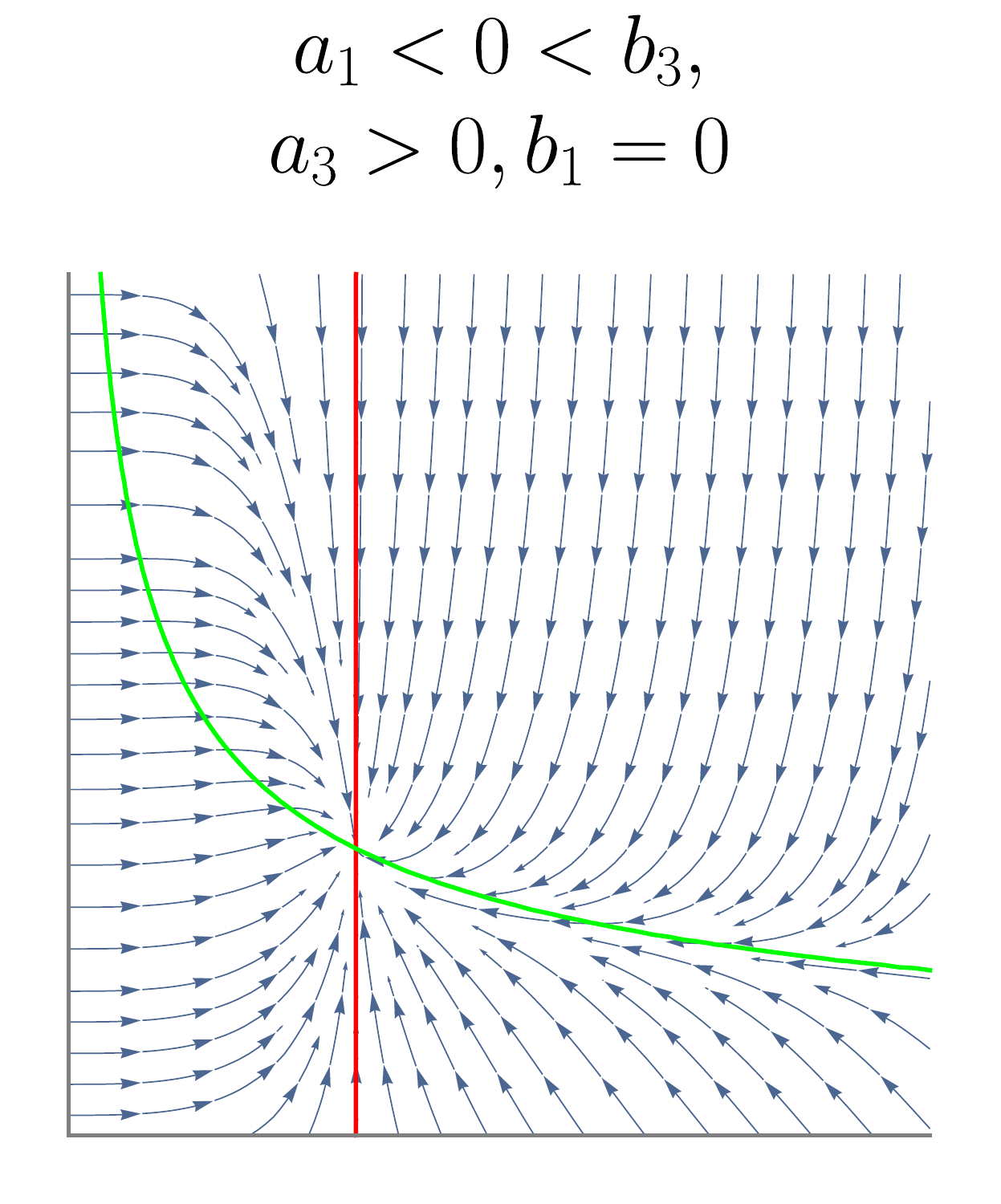} \\
\includegraphics[scale=.28]{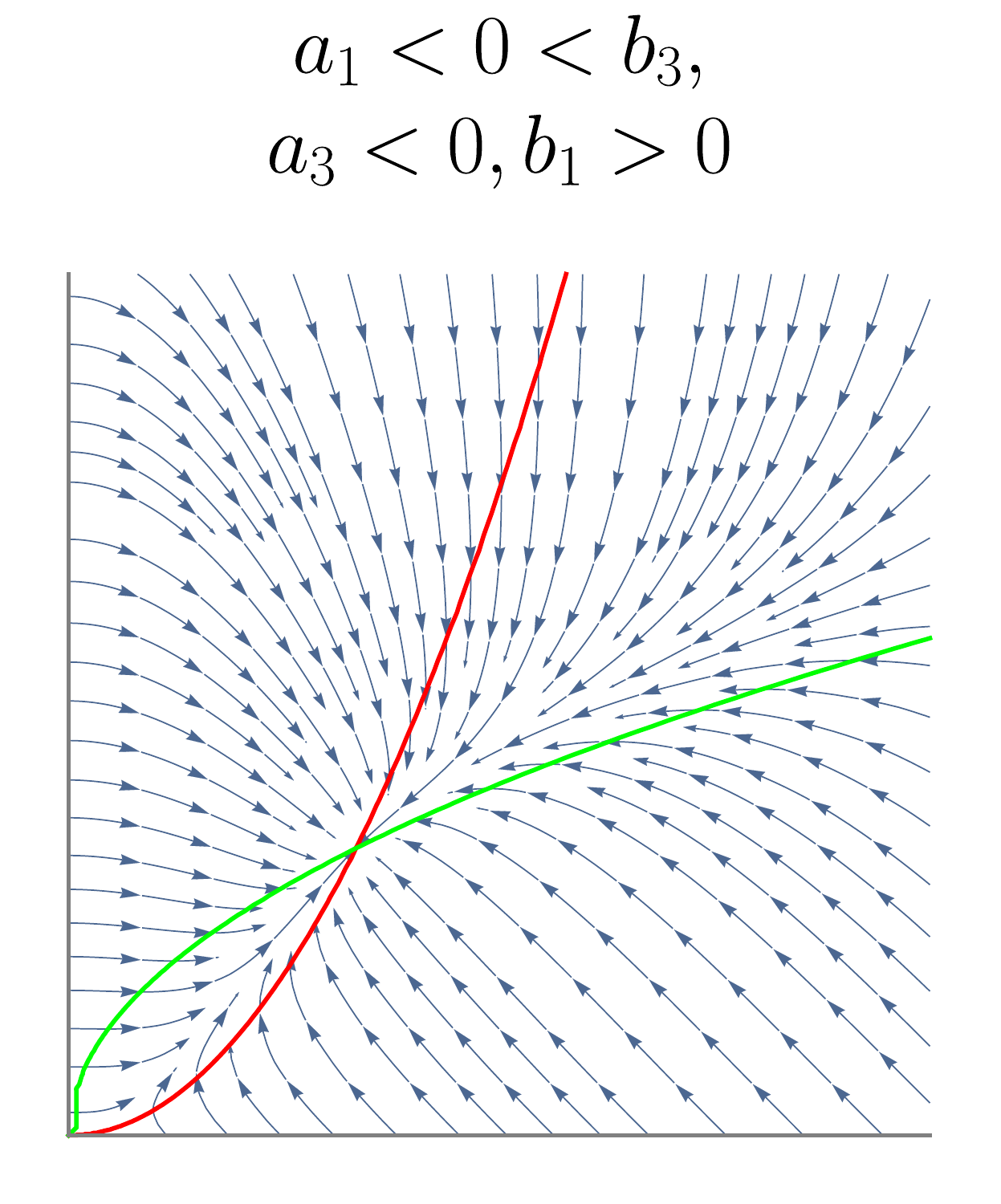} &
\includegraphics[scale=.28]{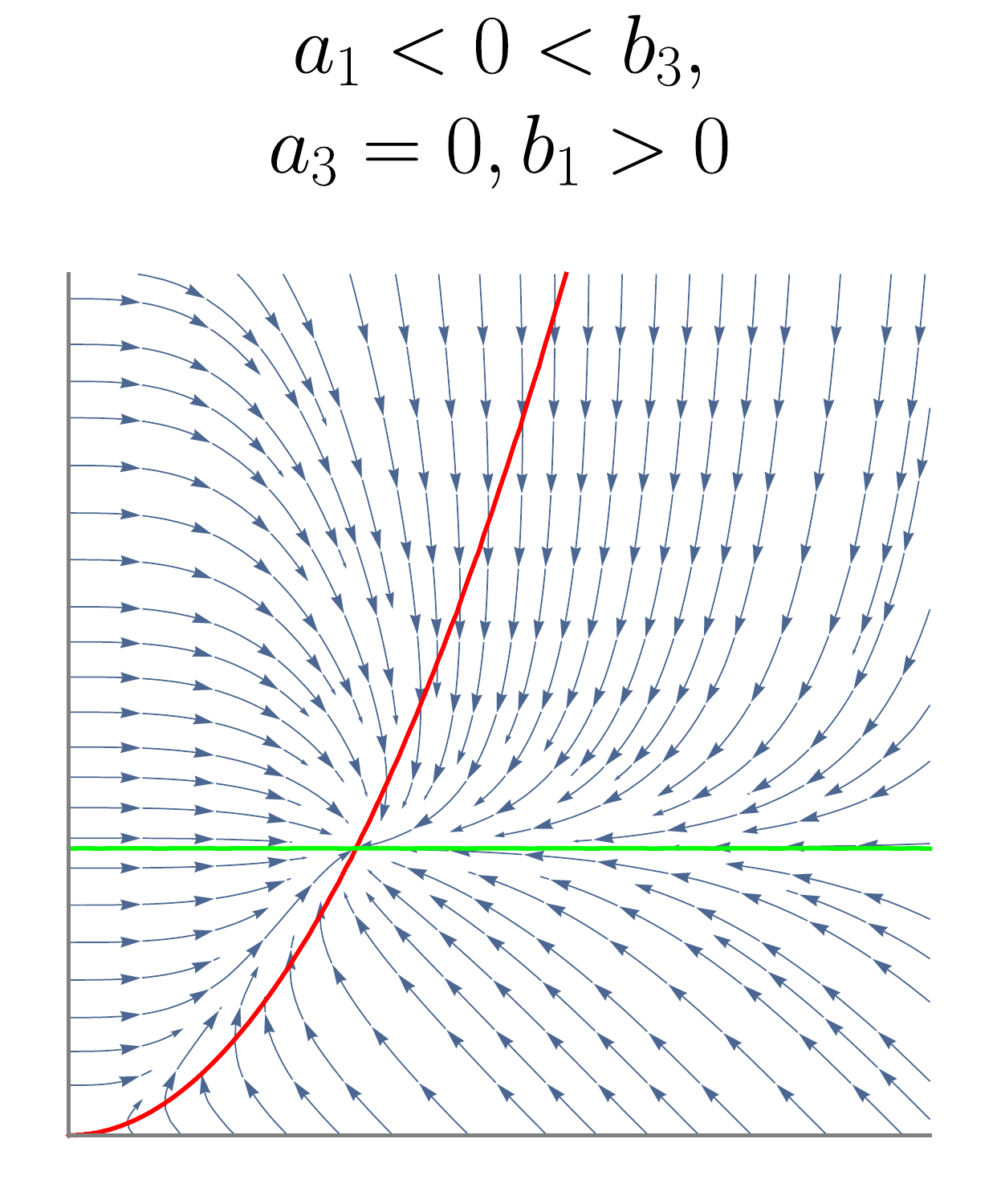} &
\includegraphics[scale=.28]{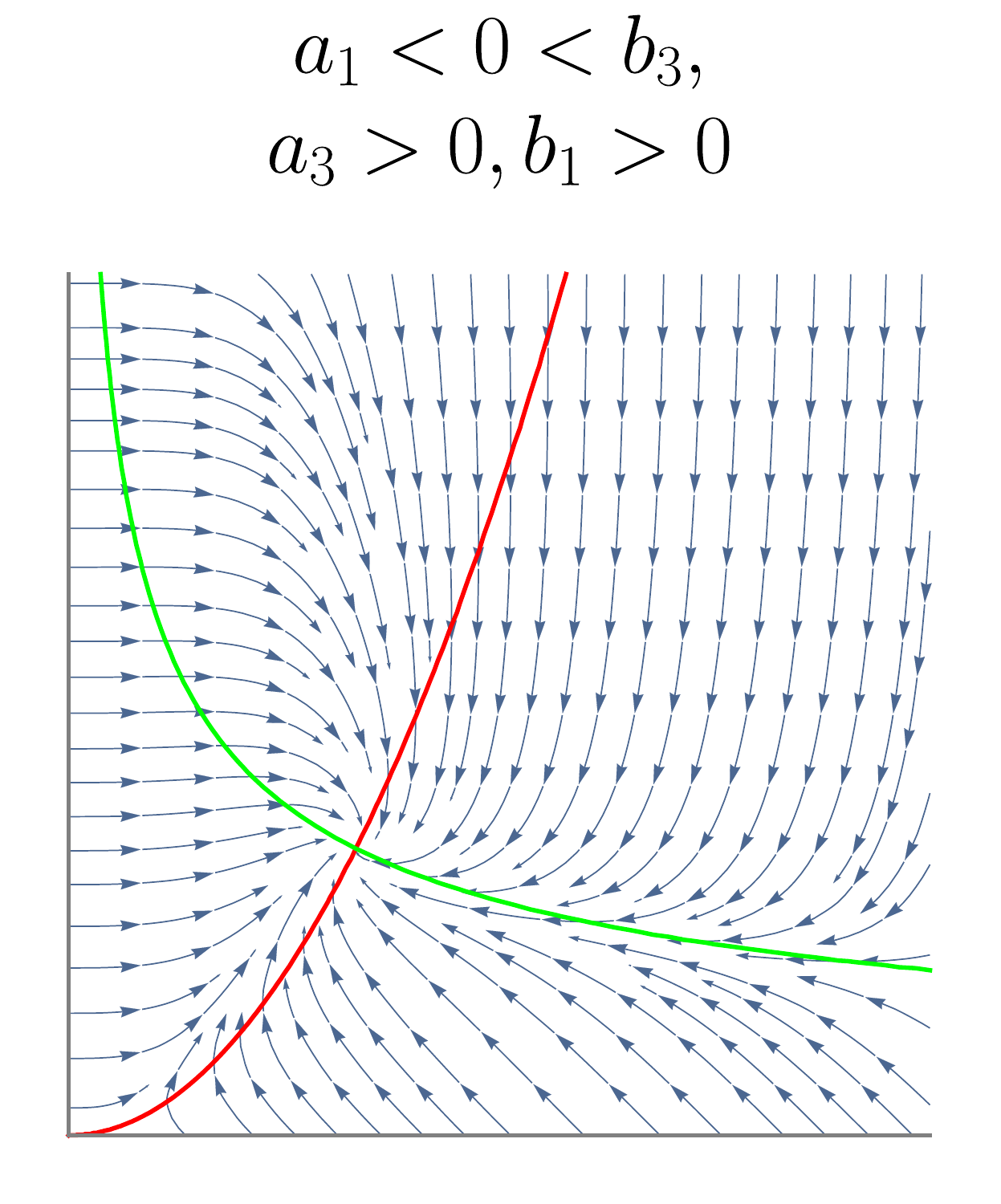} \\
\end{tabular}
\end{center}
\caption{Case $a_1 < 0 < b_3$ and $\det C < 0$. No solution approaches the boundary of the positive quadrant, and there is no unbounded solution.}
\label{fig:streamplots_very_stable_region}
\end{figure}

\begin{figure}
\begin{center}
\begin{tabular}{cc}
\includegraphics[scale=.39]{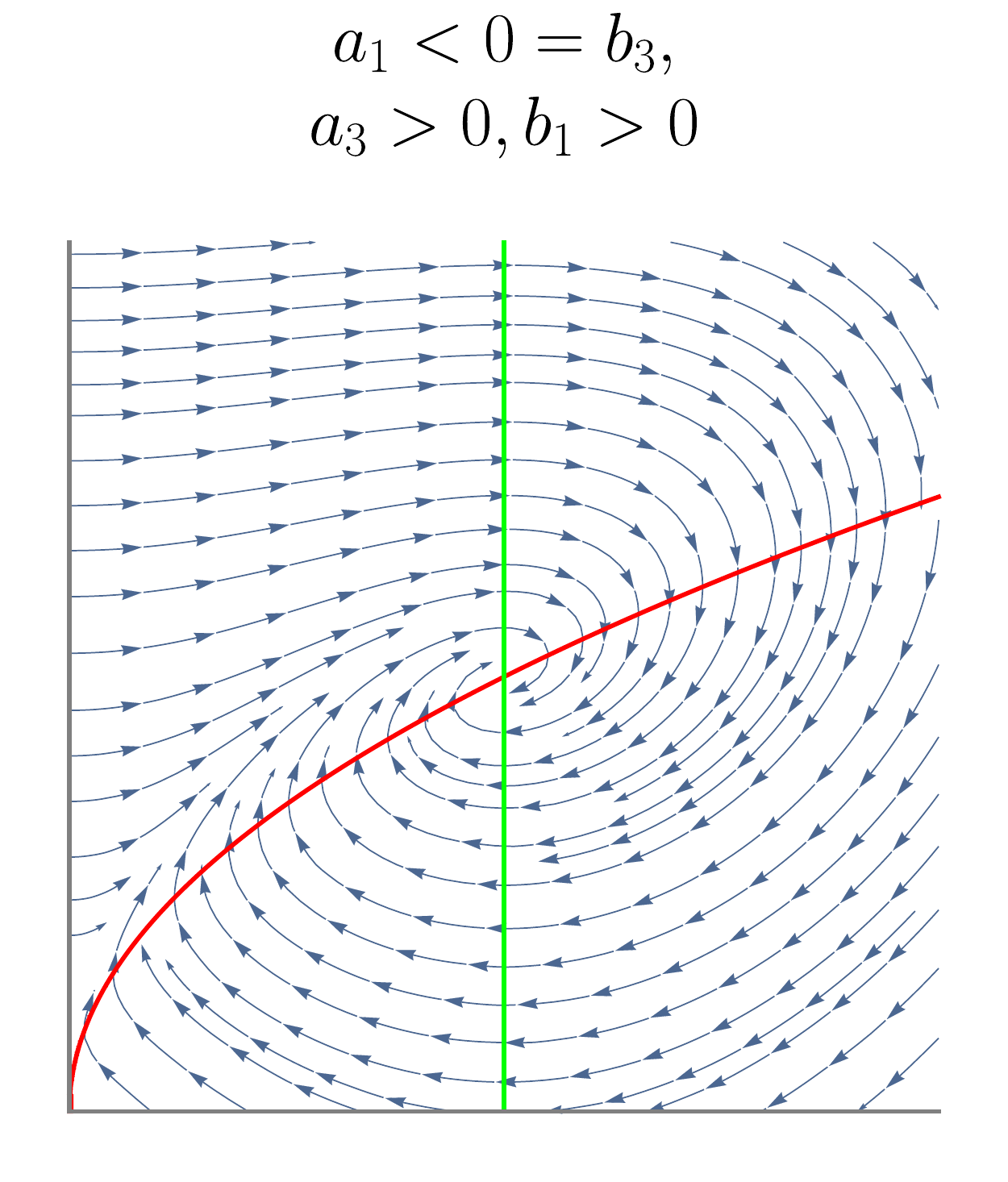} &
\includegraphics[scale=.39]{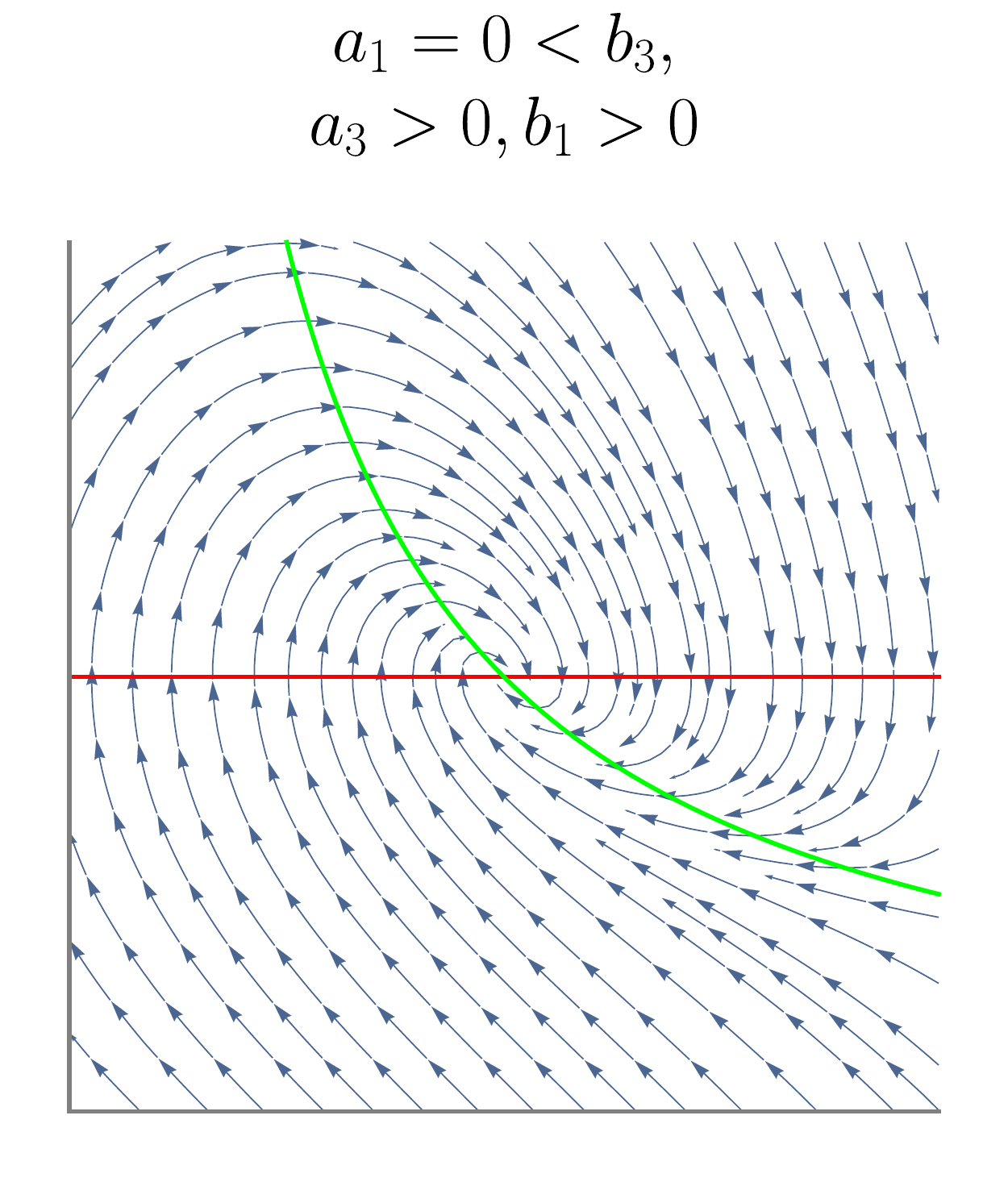} \\
\includegraphics[scale=.39]{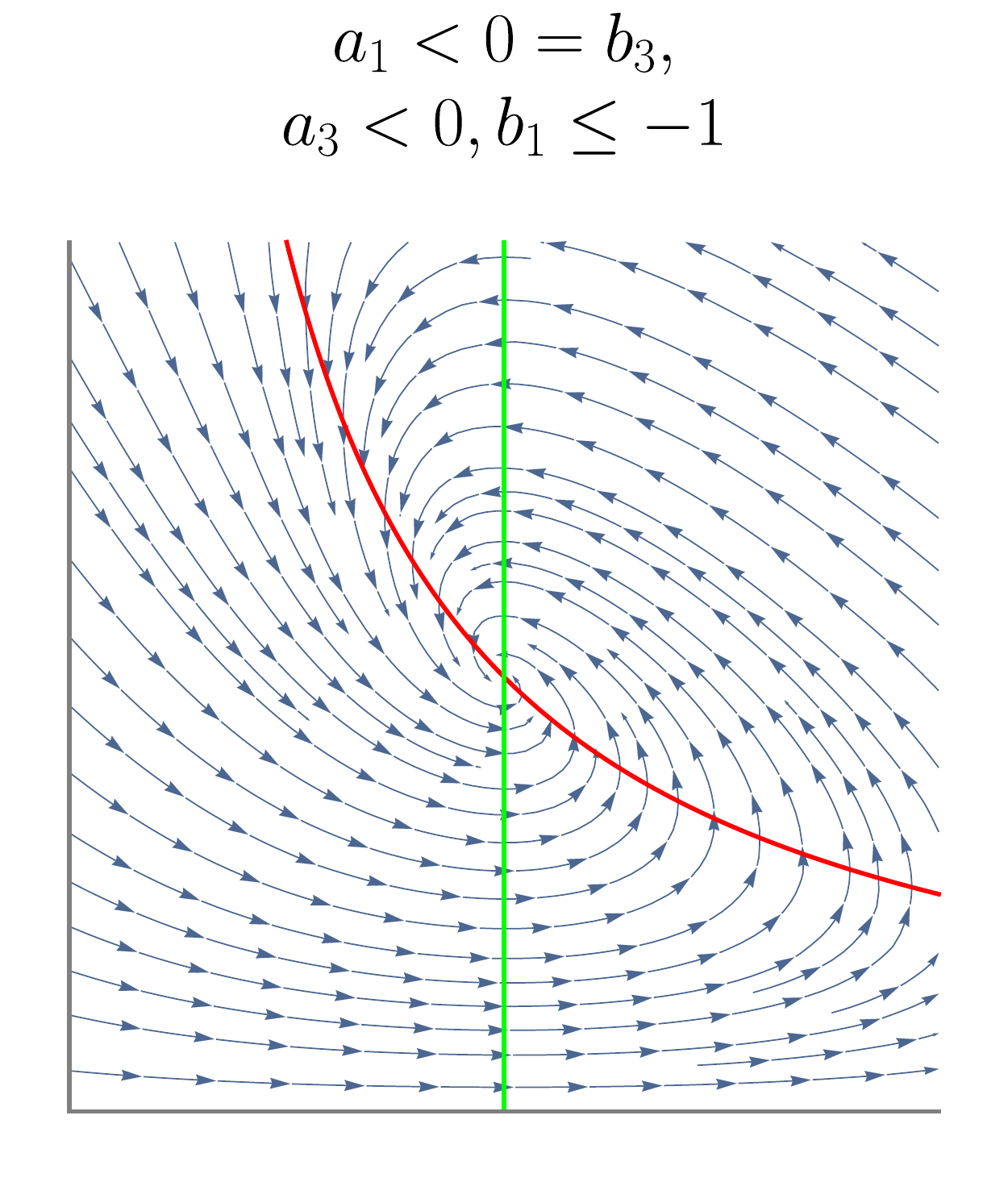} &
\includegraphics[scale=.39]{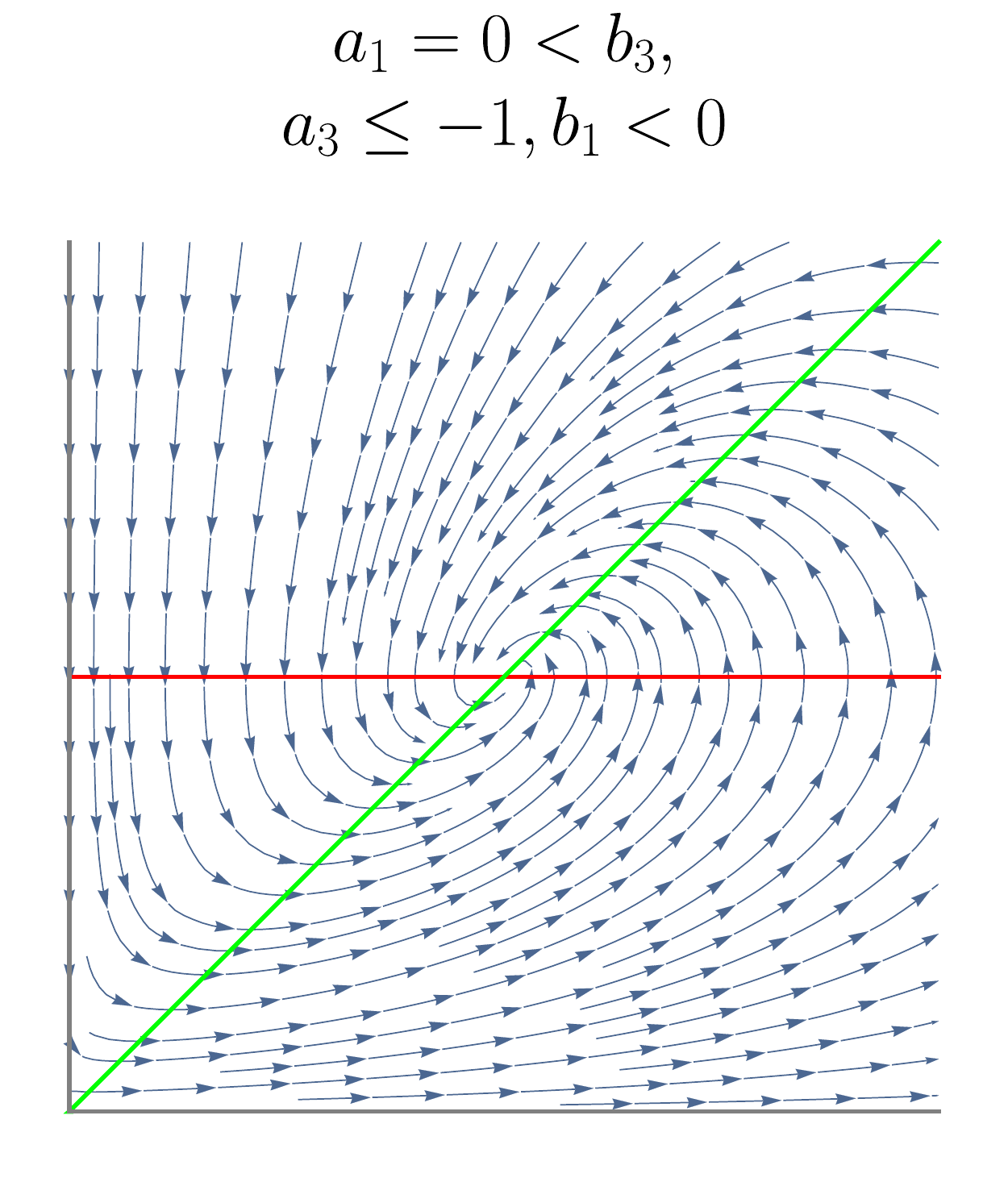} \\
\includegraphics[scale=.39]{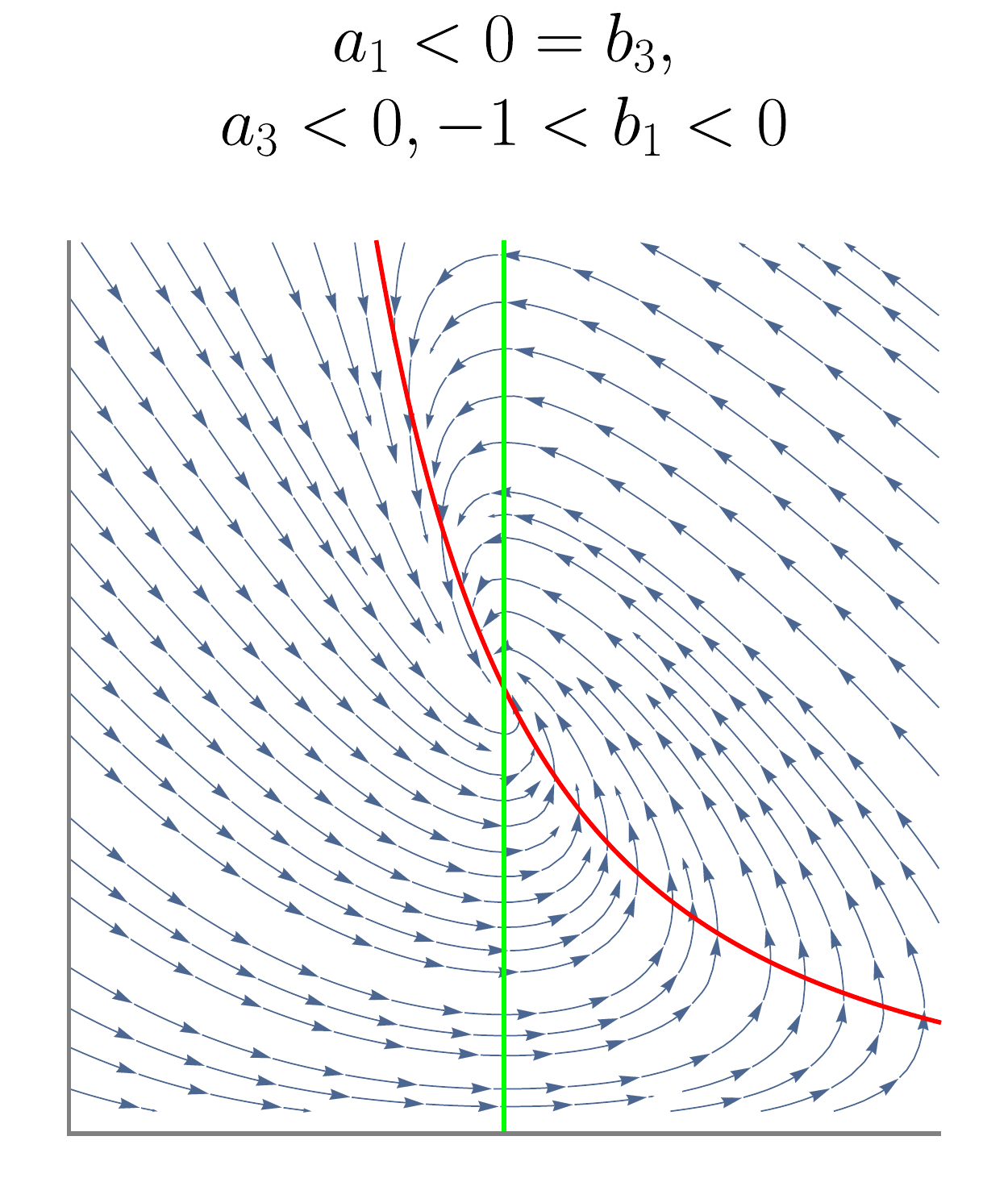} &
\includegraphics[scale=.39]{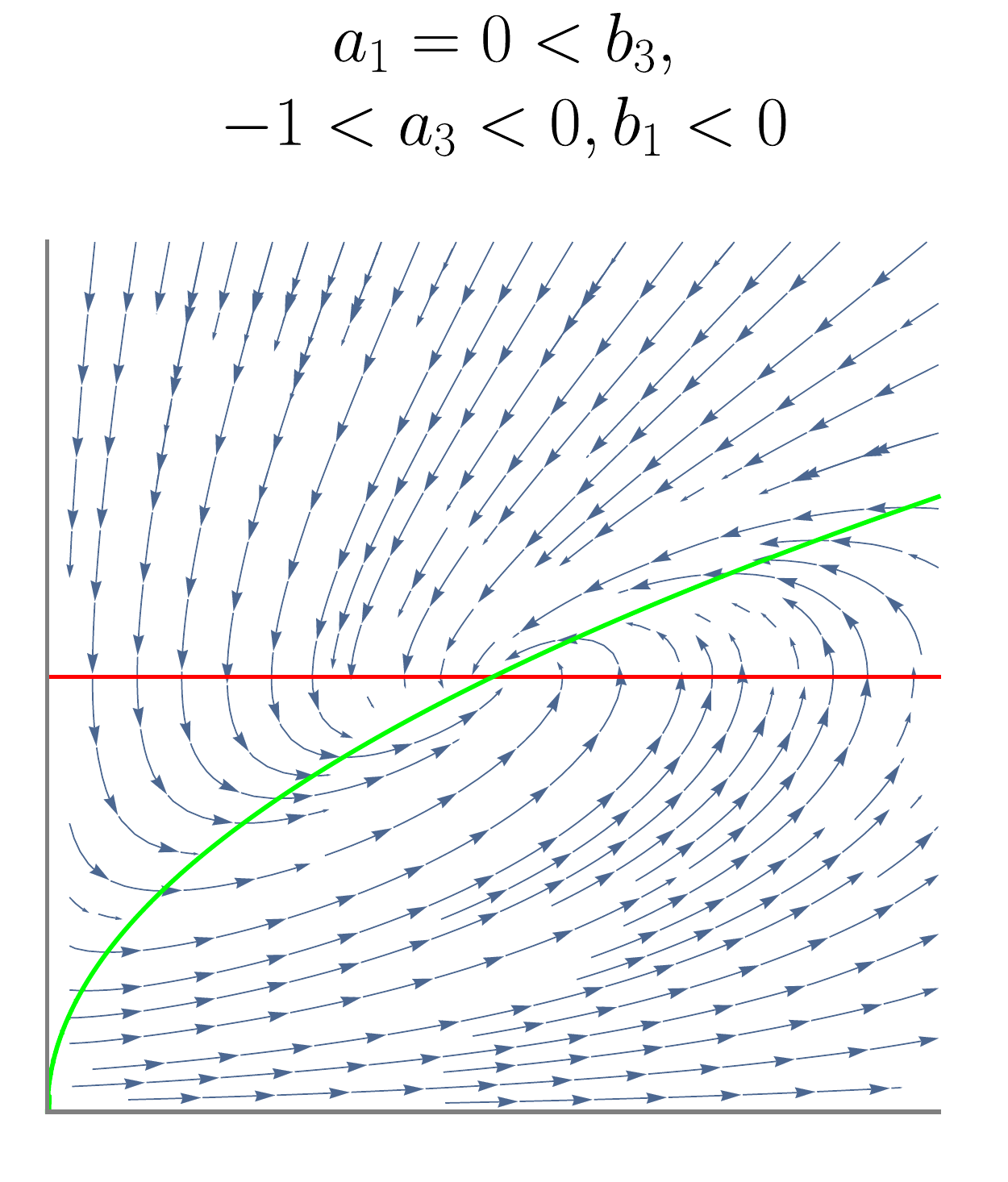}
\end{tabular}
\end{center}
\caption{Cases $a_1 < 0 = b_3$ and $a_1 < 0 = b_3$ (and $\det C < 0$). In the top and bottom row, some solutions approach the boundary of the positive quadrant, while in the middle row no solution approaches the boundary.}
\label{fig:streamplots_boundary_of_the_very_stable_region}
\end{figure}

\begin{figure}
\begin{center}
\begin{tabular}{cc}
\includegraphics[scale=.37]{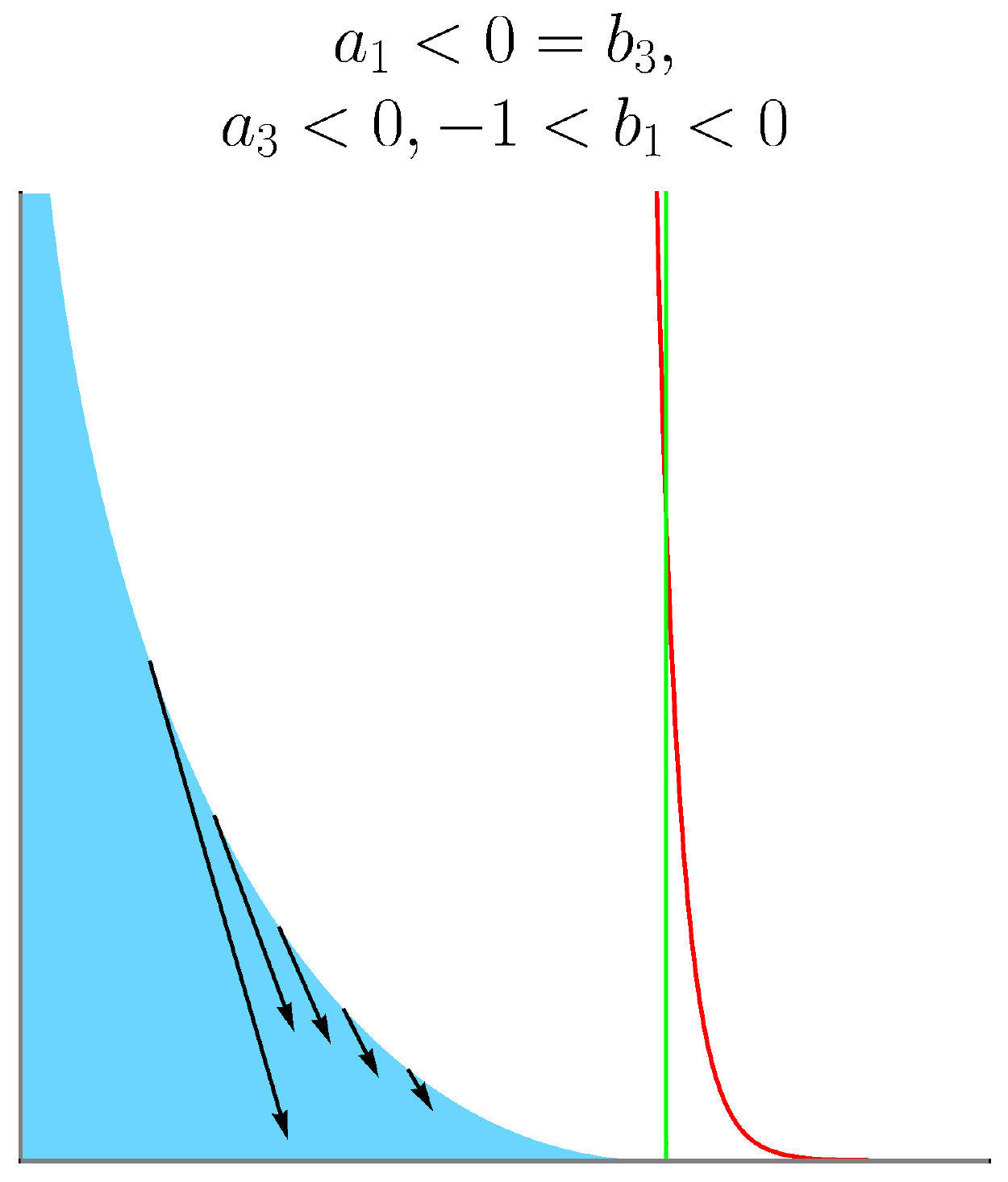} &
\includegraphics[scale=.37]{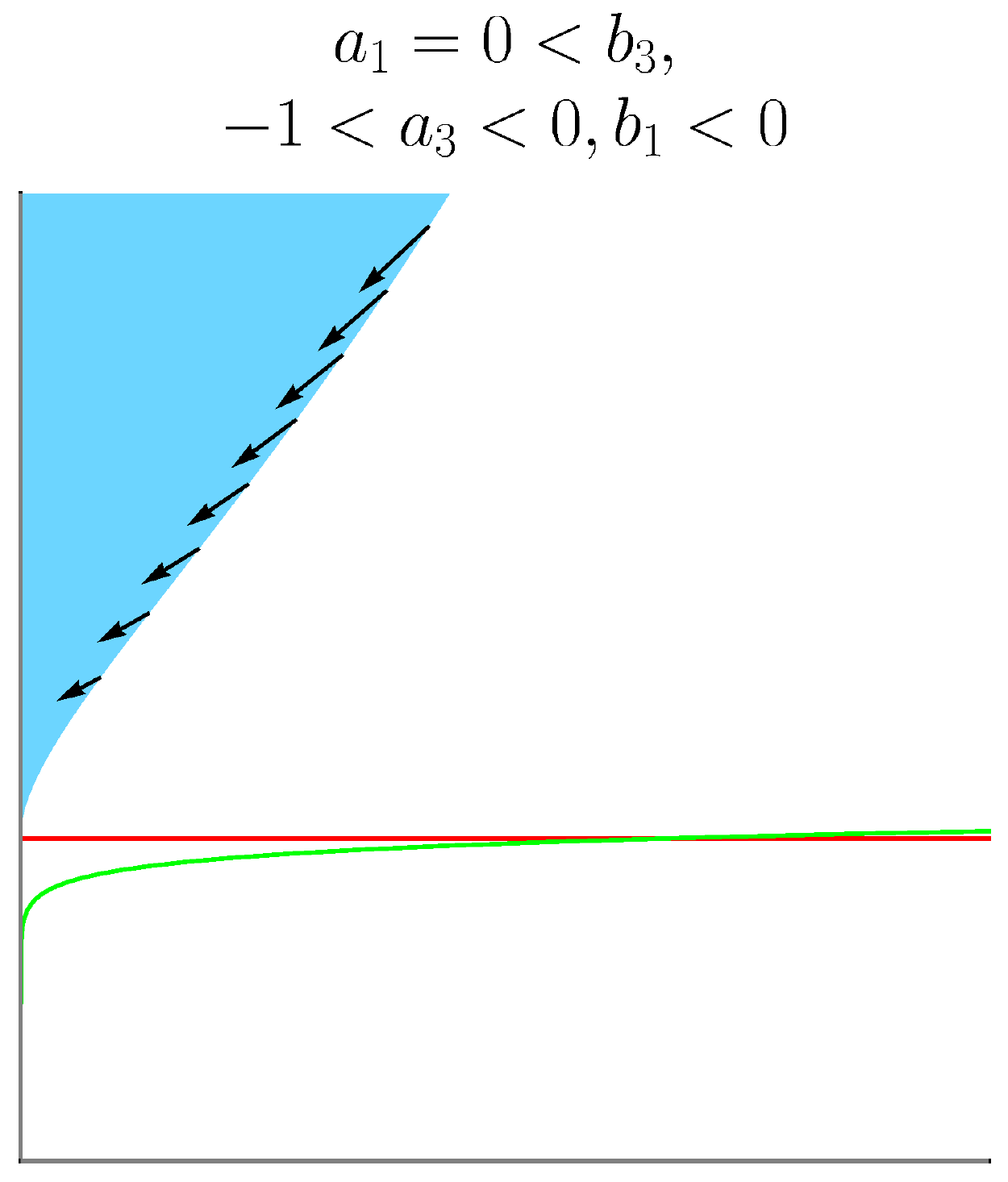}
\end{tabular}
\end{center}
\caption{The forward invariant sets constructed in Subsection \ref{subsubsec:approach_the_boundary_in_the_boundary_case}.}
\label{fig:fwd_invar_sets_boundary_of_the_very_stable_region}
\end{figure}

\begin{figure}
\begin{center}
\begin{tabular}{cc}
\includegraphics[scale=.37]{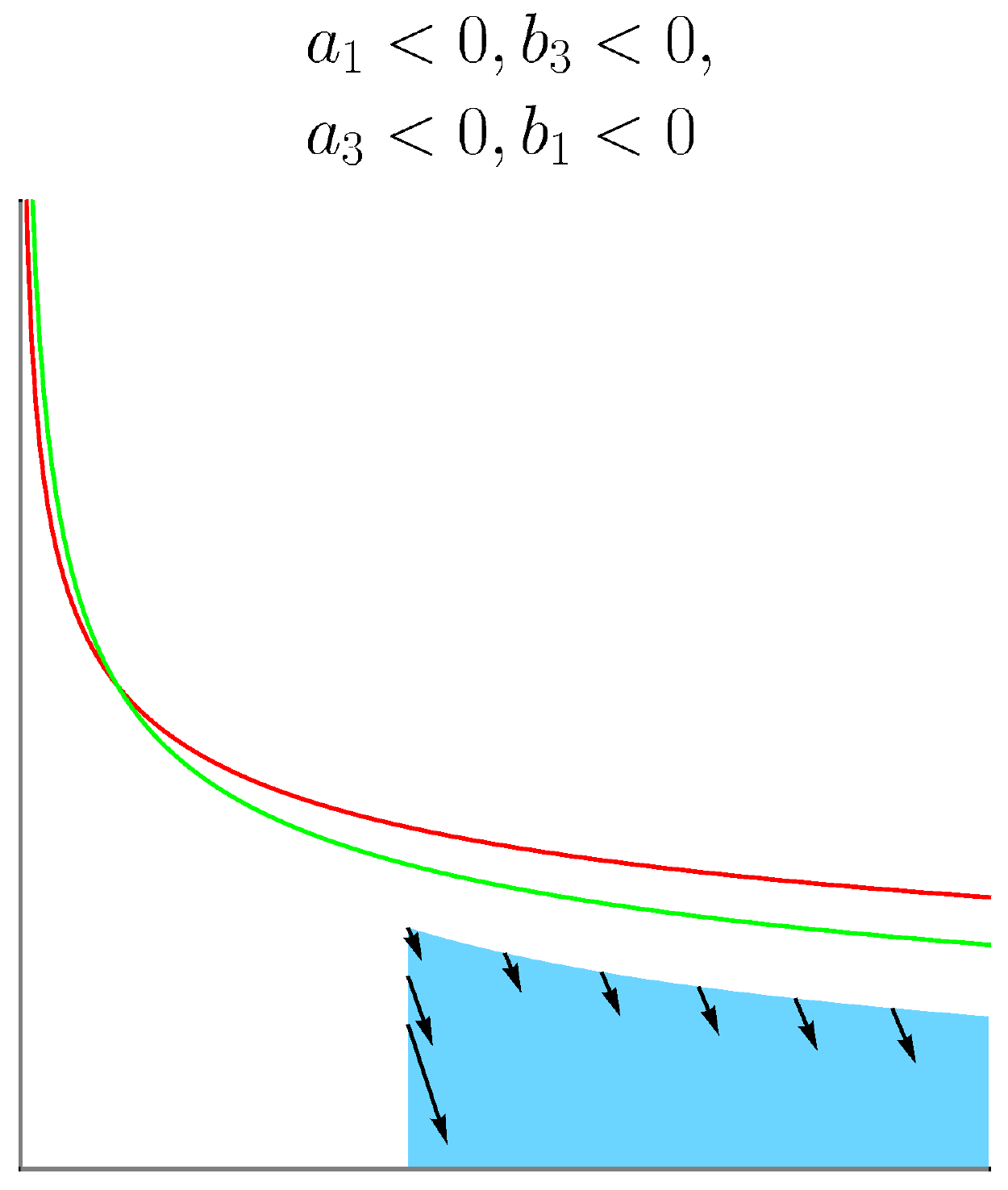} &
\includegraphics[scale=.37]{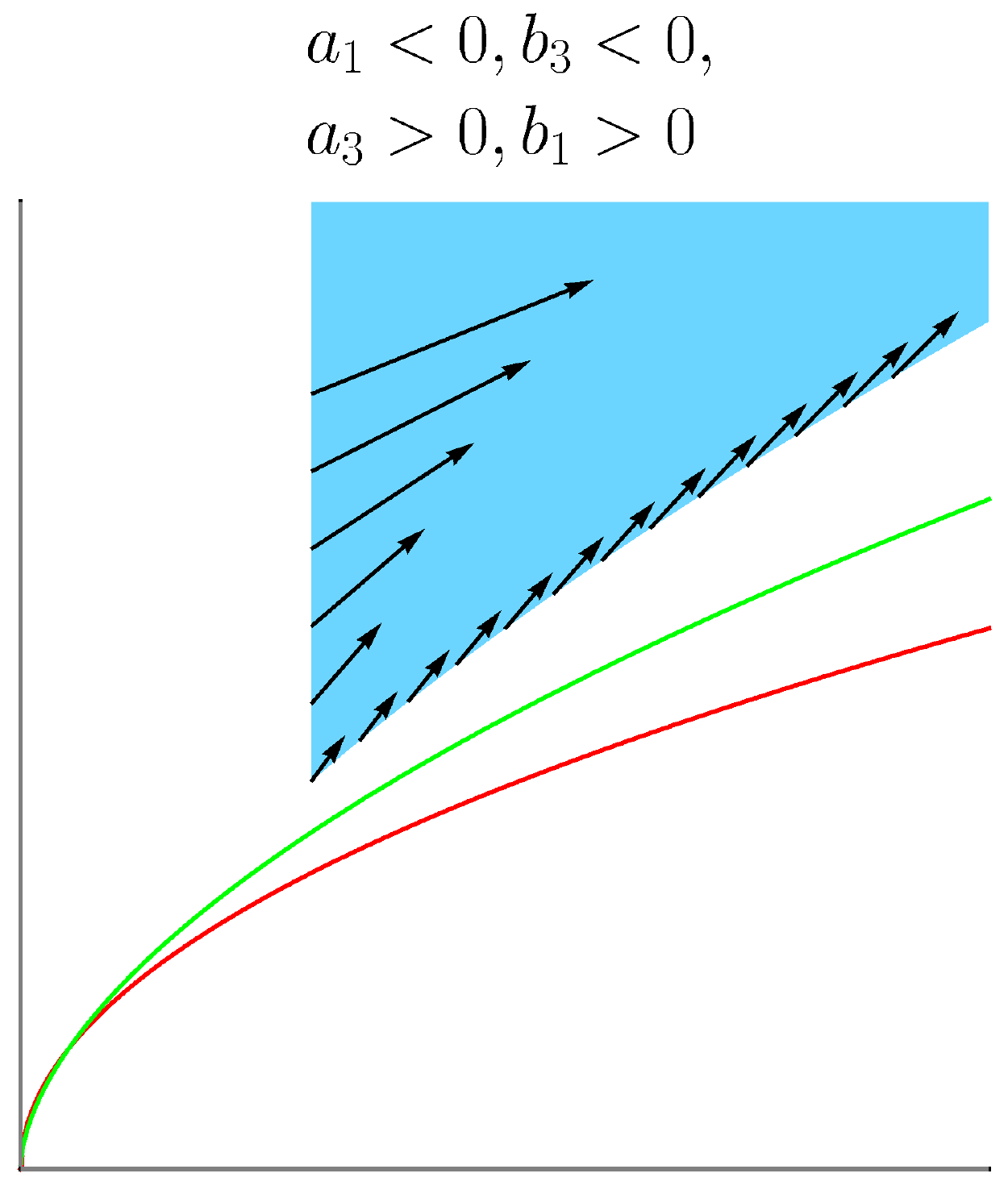}\\
\includegraphics[scale=.37]{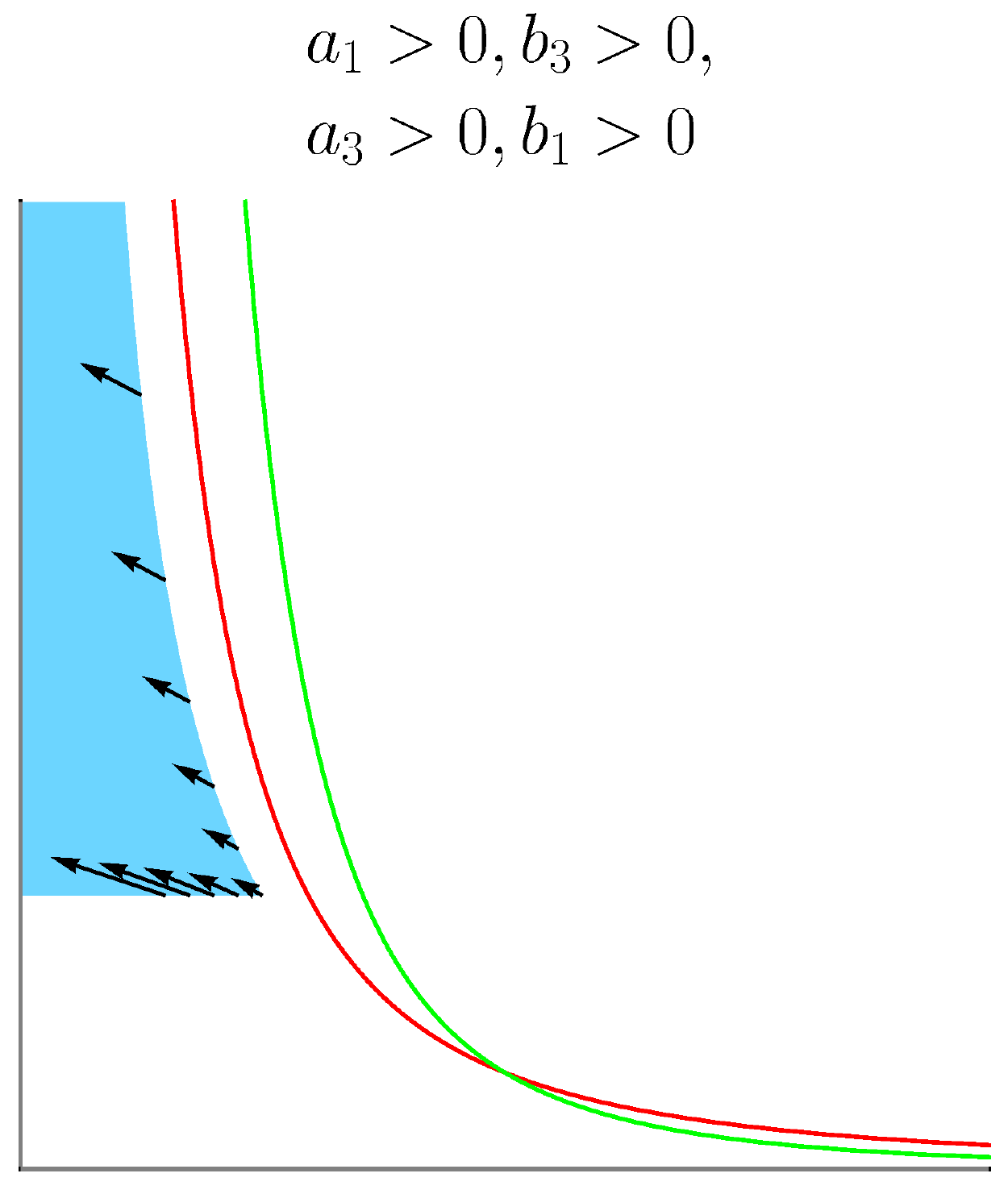} &
\includegraphics[scale=.37]{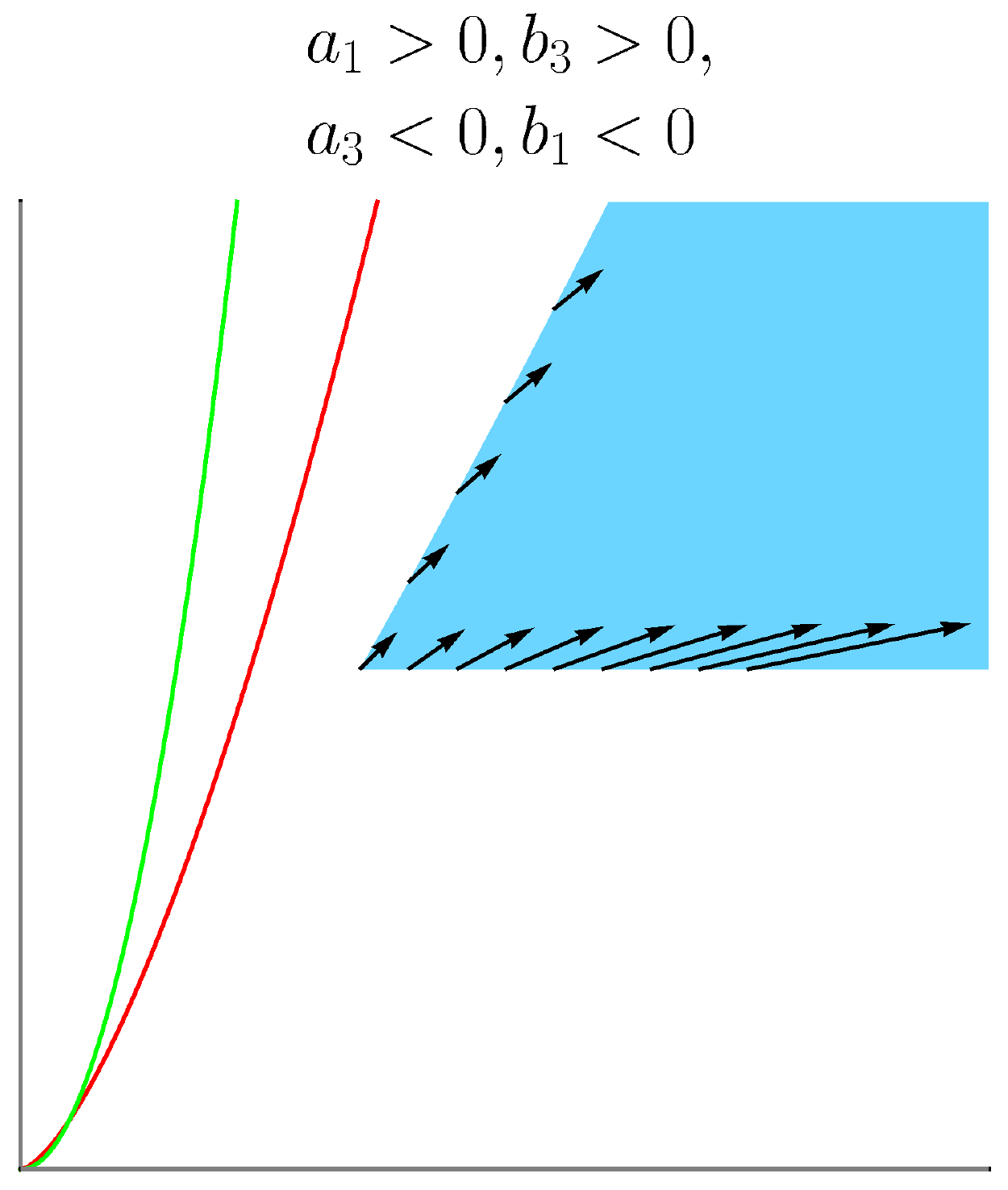}\\
\end{tabular}
\end{center}
\caption{The forward invariant sets constructed in Lemmas \ref{lemma:a1_neg_b1_neg_a3_neg_b3_neg} (top left), \ref{lemma:a1_neg_b1_pos_a3_pos_b3_neg} (top right), \ref{lemma:a1_pos_b1_pos_a3_pos_b3_pos} (bottom left), and \ref{lemma:a1_pos_b1_neg_a3_neg_b3_pos} (bottom right).}
\label{fig:fwd_invar_sets}
\end{figure}

\begin{figure}
\begin{center}
\includegraphics[scale=.44]{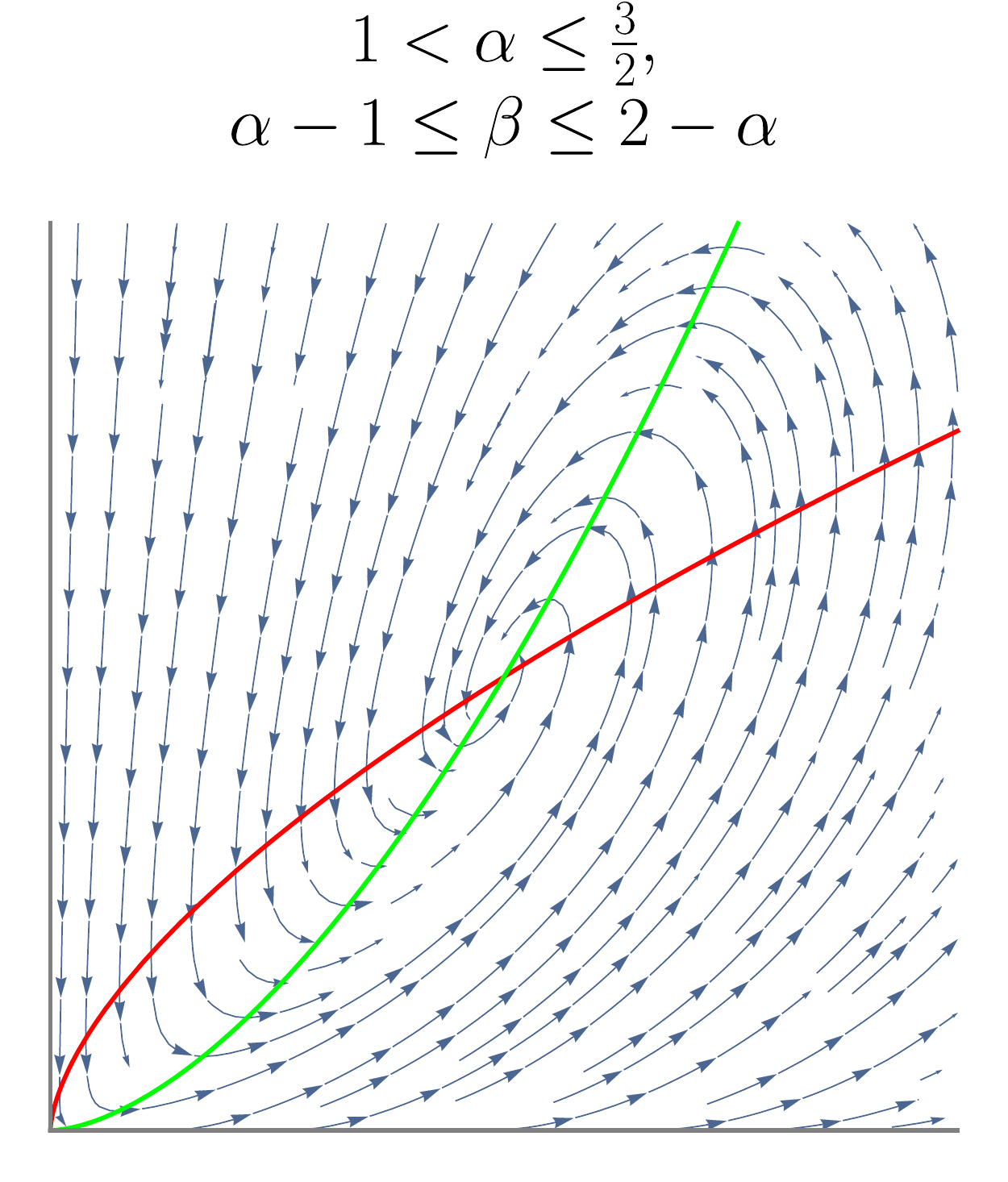}
\end{center}
\caption{Case $1<\al\leq\frac{3}{2}$ and $\al-1 \leq \be \leq 2-\al$ for the ODE~\eqref{eq:ode_two_exponents}.
Global asymptotic stability is proved in Subsection \ref{subsubsec:alpha_beta_system_glob_stab}.}
\label{fig:streamplot_triangle}
\end{figure}

\begin{figure}
\begin{center}
\begin{tabular}{cc}
\includegraphics[scale=.44]{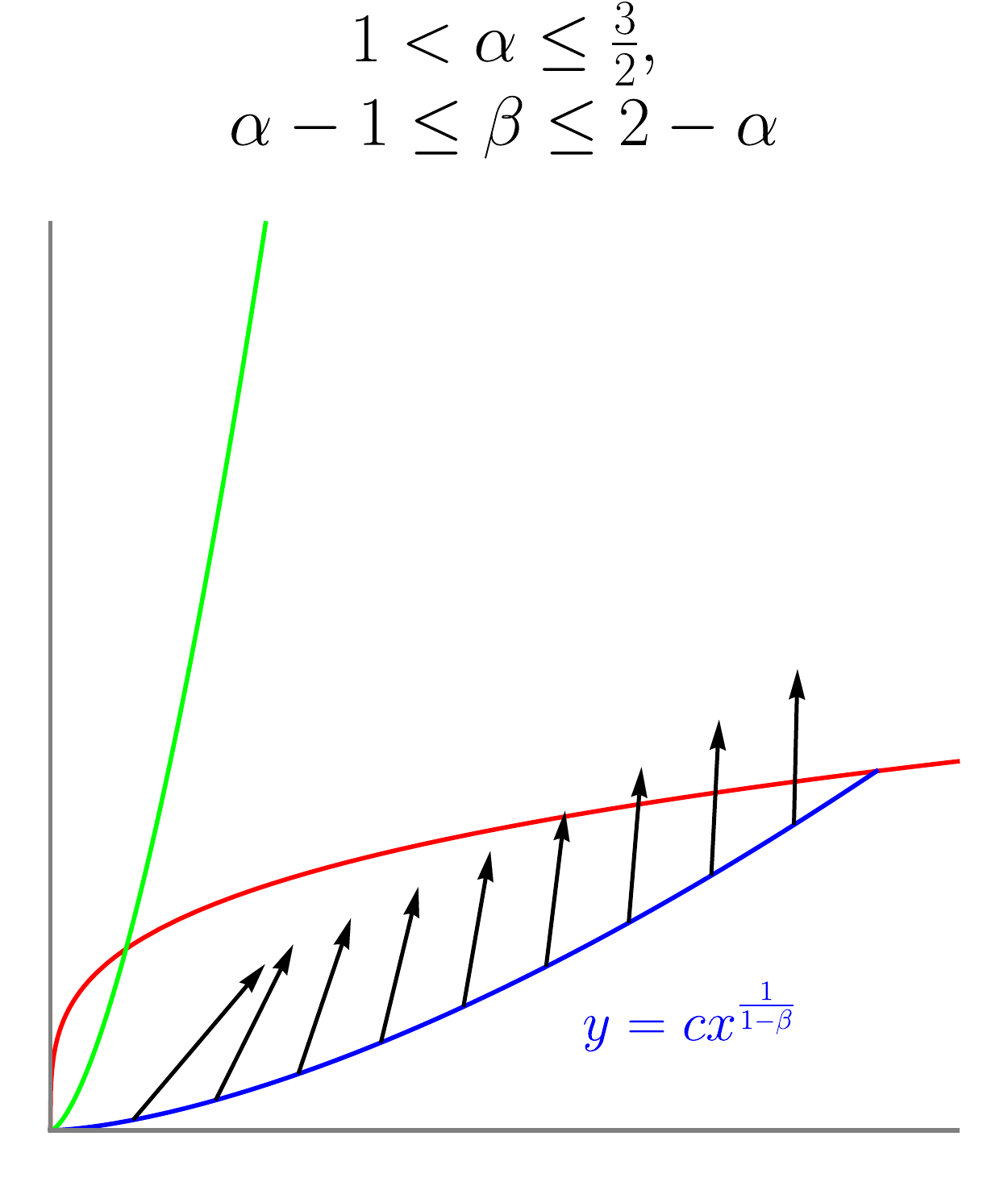} &
\includegraphics[scale=.44]{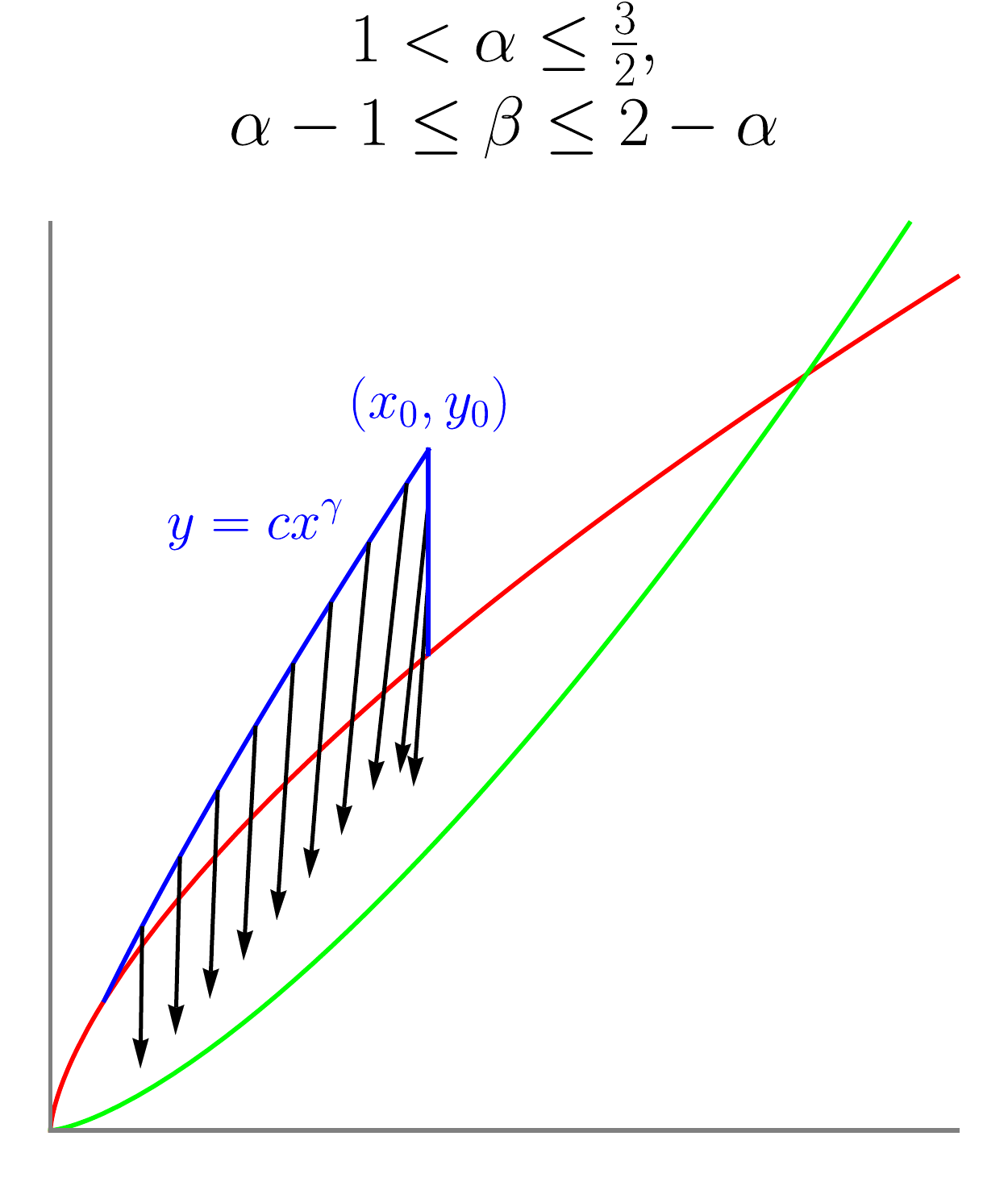}
\end{tabular}
\end{center}
\caption{Lemma~\ref{lemma:triangle_no_escape}: every solution starting below the $x$-nullcline with $x>1$ eventually reaches the $x$-nullcline (left). 
Lemma~\ref{lemma:triangle_no_approach_origin}: every solution starting above the $x$-nullcline with $y>0$ small enough eventually reaches the $x$-nullcline (right).}
\label{fig:triangle_no_escape_no_approach_origin}
\end{figure}


\end{document}